\def\focsversion{0}

\ifnum\focsversion=1
\documentclass[10pt, conference, compsocconf]{IEEEtran}

\else 
\documentclass[11pt,letterpaper]{article}
\fi

\usepackage{mathtools}
\usepackage{amsmath,amsthm,amsfonts,amssymb,bm,wasysym}
\usepackage{subfigure}
\usepackage{graphicx}
\usepackage[usenames,dvipsnames]{xcolor}
\usepackage{verbatim}
\usepackage{epsfig}
\usepackage[utf8]{inputenc}

\usepackage{float}
\usepackage[numbers]{natbib}

\usepackage{fullpage}
\usepackage{paralist}
\usepackage{hyperref}
\usepackage[ruled,algo2e]{algorithm2e}

\newtheorem{theorem*}{Theorem}
\newtheorem{proposition*}{Proposition}
\newtheorem{corollary*}{Collorary}
\newtheorem{theorem}{Theorem}[section]
\newtheorem{definition}[theorem]{Definition}

\numberwithin{equation}{section}
\newtheorem{lemma}[theorem]{Lemma}
\newtheorem{proposition}[theorem]{Proposition}
\newtheorem{coro}[theorem]{Corollary}

\numberwithin{equation}{section}

\def\bE{\operatorname*{\mathbb{E}}}
\def\B{\mathcal{B}}

\renewcommand{\phi}{\varphi}
\renewcommand{\epsilon}{\varepsilon}

\renewcommand{\emptyset}{\varnothing}

\newcommand{\E}[1]{\operatorname{\mathbb{E}}\left[#1\right]}

\newcommand{\tn}{|\kern-.1em|\kern-0.1em|}

\newcommand{\inner}[2]{\langle{#1},{#2}\rangle} 

\newcommand\be{\begin{equation}}
\newcommand\ee{\end{equation}}

\def\eps{\varepsilon}

\DeclarePairedDelimiter\floor{\lfloor}{\rfloor}

\ifnum\focsversion = 1
\author{\IEEEauthorblockN{Christian Borgs and Jennifer Chayes}
\IEEEauthorblockA{Microsoft Research.\\
Cambridge, MA\\
 \{borgs, jchayes\}@microsoft.com}
\and
\IEEEauthorblockN{Adam Smith}
\IEEEauthorblockA{Boston University\\
Boston, MA\\
ads22@bu.edu
}
\and
\IEEEauthorblockN{Ilias Zadik}
\IEEEauthorblockA{MIT\\
Cambridge, MA\\
izadik@mit.edu}
}

\else 
\author{
Christian Borgs\thanks{Microsoft Research New England. {\tt \{Christian.Borgs,Jennifer.Chayes\}@microsoft.com}. }
\and
Jennifer Chayes\footnotemark[1]
\and
Adam Smith\thanks{Boston University. {\tt ads22@bu.edu}. }
\and
Ilias Zadik\thanks{MIT. {\tt izadik@mit.edu}. Research done in part while an intern at Microsoft Research New England. }
}
\fi

\begin{document}

\title{Revealing Network Structure, Confidentially:\\ Improved Rates for Node-Private Graphon Estimation }
\date{}

\maketitle

\begin{abstract}
  Motivated by growing concerns over ensuring privacy on social
  networks, we develop new algorithms and impossibility results for
  fitting complex statistical models to network data subject to
  rigorous privacy guarantees. We consider the so-called
  \emph{node-differentially private} algorithms, which compute
  information about a graph or network while provably revealing almost
  no information about the presence or absence of a particular node in
  the graph.

  We provide new algorithms for node-differentially private estimation
  for a popular and expressive family of network models:
  \emph{stochastic block models} and their generalization,
  \emph{graphons}. Our algorithms improve on prior work
  \cite{BorgsCS15}, reducing their error quadratically and
  matching, in many regimes, the optimal nonprivate
  algorithm~\cite{KloppTV15}.  We also show that for the simplest random
  graph models ($G(n,p)$ and $G(n,m)$), node-private algorithms can be
  qualitatively more accurate than for more complex
  models---converging at a rate of $\frac 1 {\eps^2 n^{3}}$ instead of
  $\frac 1 {\eps^2 n^2}$. This result uses a new extension lemma for
  differentially private algorithms that we hope will be broadly useful.
\end{abstract}

\ifnum\focsversion = 1
\begin{IEEEkeywords}
  Differential privacy, stochastic block models, graphons, private
  data analysis.
\end{IEEEkeywords}
\fi

\section{Introduction}

Network data play an increasingly important role in many
scientific fields. Data from social networks, in which the nodes
represent individuals and edges represent relationships among them,
are  transforming sociology, marketing, and political science, among
others. However, what makes these data so valuable also makes them
highly sensitive---consider, for example, the public sentiment
surrounding the recent Cambridge Analytica scandal. 

What kinds of information can we release about social networks while
preserving the privacy of their users? Straightforward approaches,
such as removing obvious identifiers or releasing summaries that
concern at least a certain number of nodes, can be easily
broken~\cite{NS09,Korolova10}. 

In this paper, we develop new algorithms and impossibility results for
fitting complex statistical models to network data subject to rigorous
privacy guarantees. We consider \emph{differentially private}
algorithms~\cite{DMNS06}. There are two main variants of differential privacy for
graphs: \emph{edge} and \emph{node} differential privacy~\cite{RHMS09}. Intuitively,
edge differential privacy ensures that an algorithm's output does not
reveal the inclusion or removal of a particular edge in the graph,
while node differential privacy hides the inclusion or removal of a
node together with all its adjacent edges. Edge privacy is weaker
(hence easier to achieve) and has been studied more
extensively \citep{NRS07,RHMS09,KRSY11,MirW12,LuM14exponential,KarwaSK14,HayLMJ09,HayRMS10,KarwaS12,LinKifer13,KarwaS15,GRU12,BlockiBDS12jl,MirW12,KarwaSK14,LuM14exponential,KarwaS15,XiaoCT14}.


We study node-differentially private algorithms. These ensure that, no
matter what an analyst observing the output knows ahead of time, she
learns the same things about an individual Alice regardless of whether
Alice's data are used or not. Node privacy's stringency makes
the design of accurate, private algorithms challenging; only a small
number of techniques for designing such algorithms are known~\cite{KNRS13,BBDS13,ChenZ13,RaskhodnikovaS16,DayLL16}.

We provide new algorithms for node-differentially private estimation
for a popular and expressive family of network models: \emph{stochastic block
models} and their generalization, \emph{graphons}. Our algorithms
improve on prior work (by a subset of us~\cite{BorgsCS15}), roughly
reducing their error quadratically and matching, in many regimes, the optimal
nonprivate algorithm~\cite{KloppTV15,McMillanS17}.  We also show that for the simplest random graph
models ($G(n,p)$ and $G(n,m)$), node-private algorithms can be qualitatively
more accurate than for more complex models---converging at a rate of
$\frac 1 {\eps^2  n^{3}}$ instead of $\frac 1 {\eps^2 n^2}$. This result uses a new
extension lemma for differentially private algorithms that we hope
will be broadly useful.

\paragraph{Modeling Large Graphs via Graphons}
Traditionally, large graphs have been modeled using various parametric models,
one of the most popular being the \emph{stochastic block model}
\cite{HLL83}.
Here one
postulates that an observed graph was generated by first assigning
vertices at random to one of $k$ groups, and then
connecting two vertices with a probability that depends on the
groups the two vertices are members of.

As the number of vertices of the graph in question grows, we
  do not expect the graph
 to be well described by a stochastic block model with a fixed
number of blocks.  We therefore consider nonparametric models
described by a
\emph{graphon}. A graphon is a measurable, bounded function $W:[0,1]^2\to
[0,\infty)$ such that $W(x,y)=W(y,x)$, which for convenience we take
to be normalized: $\int W=1$.  Given a graphon, we
generate a graph on $n$ vertices by first assigning i.i.d. uniform
labels $x_i \in [0,1], i=1,2,\ldots,n$ to the vertices, and then connecting vertices $i,j$ with
labels $x_i,x_j$ with probability $H_n(i,j)=\rho_n W(x_i,x_j)$, where $\rho_n$ is a
parameter determining the density of the generated graph $G_n$
with $\rho_n\|W\|_\infty\leq 1$.   We call $G_n=G_n(\rho W)$ a
$W$-random graph with target density $\rho_n$ (or simply a
$\rho_nW$-random graph).

This model captures stochastic block models as well as more
complex models, e.g. random geometric graphs, where each vertex
corresponds to a point in a metric space (selected randomly according
to a particular distribution) and vertices share an
edge if their points are sufficiently close
\cite{Gilbert61,DallC02,Penrose03,Galhotra18}.

For both the ``dense'' setting (where the target density $\rho_n$ does not
depend on the number of vertices) and the ``sparse'' setting (where
$\rho_n\to 0$ as $n\to\infty$), graphons play a key role in the
convergence theory for graph sequences
\cite{H79,A81,LS06,BCLSV06,BCLSV08,BCLSV12,BCCZ14a,BCCZ14b}, providing
limit objects in several natural topologies.

\paragraph{Metrics for Estimation} Given a single graph $G_n$
generated as $\rho W$-random for unknown $\rho$ and $W$, how
well can we estimate $\rho$ and $W$? This task has now been studied
extensively \cite{BC09,RCY11,CWA12, BCL11,LOGR12, 
 TSP13, 
 LR13, 
 WO13, 
 CA14,
 ACC13, YangHA14, 
 GaoLZ14, 
ABH14,
 Chatterjee15,
AS15unknown,KloppTV15,McMillanS17}.
One issue faced by all these works is \emph{identifiability}: multiple
graphons can lead to the same distribution on $G_n$.  Specifically,
two graphons $W$ and $\tilde W$ lead to the same distribution on
$W$-random graphs if and only if there are measure preserving maps
$\phi,\tilde\phi:[0,1]\to[0,1]$ such that
$W^\phi=\widetilde W^{\widetilde \phi}$, where $W^\phi$ is defined by
$W(x,y)=W(\phi(x),\phi(y))$ \cite{DJ08,BCL10}.  Hence, there is no
``canonical graphon'' that an estimation procedure can output.  Some
of the literature circumvents identifiability by making strong
additional assumptions that imply the existence of canonical
equivalence class representatives. We make no such assumptions, but
instead define consistency in terms of a metric on equivalence
classes.  We use a variant of
the $L_2$ metric,
 \begin{equation}
   \label{delta-def}
\delta_2(W,W')=\inf_{\phi:[0,1]\to[0,1] }
 \|W^\phi-W'\|_2\, ,
\end{equation}
where $\phi$ ranges over measure-preserving bijections.

In this work, we set aside questions of computational efficiency and
focus on establishing what rates are possible in principle (our
algorithms, like the nonprivate state of the art, run in time
	roughly exponential in $n$). 

For our purposes, the most relevant work is that of Klopp, Tzybakov
and Verzalen~\cite{KloppTV15}, who establish tight upper and (in parallel
to~\cite{McMillanS17}) lower bounds on the error rate of nonprivate
algorithms, given a single $n$-vertex $\rho
W$-random graph and a target number of blocks, $k$. Our algorithms match their rate for large enough
values of the privacy parameter. 

\paragraph{Private Algorithms for Graph Data and the Rewiring Metric} 
Let $\mathcal{A}$ be a randomized algorithm that takes values from
some input metric space $(\mathcal{M},d)$ (called the space of
\emph{data sets}) and ouputs probability distributions on some measurable space $(\Omega, \mathcal{F})$.
\begin{definition}
The algorithm $\mathcal{A}$ is $\epsilon$-differential private
($\epsilon$-DP) with respect to the metric $d$ if, for all subsets $S \in \mathcal{F}$ and $D_1,D_2 \in \mathcal{M}$, $$\mathbb{P}\left(\mathcal{A}(D_1) \in S\right) \leq \exp\left[\epsilon d(D_1,D_2)\right]\mathbb{P}\left(\mathcal{A}(D_2) \in S\right).$$
\end{definition}

The metric $d$ is typically defined by specifying pairs of data sets that
are adjacent (i.e., at distance 1 from each other), and then letting $d$
be the induced path metric. 

There are two natural variants of differential privacy suited for
graph datasets, edge differential privacy and node differential
privacy. Intuitively, edge differentially private algorithms hide the
presence or absence of a particular relationship between individuals
in a social network, while node differentially private algorithms
protect each individual together with all his/her relationsips. In
both cases, the data set is an undirected graph with no self-loops; we
let $\mathcal{G}_n$ denote the set of such graphs on $n$ vertices. 
Formally, edge differential privacy is obtained by taking $d$ to count
the number of edges that differ between two graphs (the Hamming metric
on adjacency matrices).  
%
%
In contrast, node differential privacy is defined with respect to the \textit{rewiring} metric, or \emph{node distance}, between graphs:
we say that two distinct graphs  $G,G'$ are at
node-distance 1 (or \emph{adjacent}) if one can be obtained from
the other by inserting or removing arbitrary sets of edges adjacent to a singe vertex, a process we call \textit{rewiring} the vertex. For arbitrary $G_1,G_2 \in \mathcal{G}_n$, define the \textbf{node-distance} between them, $d_v(G_1,G_2)$, to be the minimum number of vertices of $G_1$ that need to be rewired to obtain $G_2$. 
A randomized algorithm $\mathcal{A}$ defined on $\mathcal{G}_n$ is  $\eps$-node differentially private ($\epsilon$-node DP) if it is $\epsilon$-differentially private with respect to the node-distance $d_v$.

Edge differential privacy is a weaker notion and has been extensively
studied over the past decade. Algorithms have been developed for various tasks such as the release of subgraph counts, the degree distribution and the parameters of generative graph models \cite{HayLMJ09}, \cite{KarwaSK14}, \cite{MirW12}, \cite{KarwaS12}, \cite{KRSY11}, \cite{NRS07}. On the other hand, the node-differential privacy is a much stronger privacy guarantee. The first nontrivial node-differentially algorithms were designed
(concurrently) in~\cite{BBDS13,ChenZ13,KNRS13}, with a focus on
algorithms that release one-dimensional summaries of a network such as
subgraph counts. Later work~\cite{RaskhodnikovaS16,BorgsCS15,DayLL16}
introduced higher-dimensional techniques. Most relevant here, a subset of
us gave the first algorithms for node-private graphon
estimation~\cite{BorgsCS15}. A common thread to all these works is the
use of \emph{Lipschitz extensions} in the rewiring metric to control
the sensitivity of summary statistics for sparse graphs. A key piece
of this paper is a novel use of such extensions. 

The previous results for graphon estimation achieved estimation
error going to 0 for a large parameter range, but fell short in several
respects: first, even when $\eps$ is arbitrarily large, the algorithm
does not match the best nonprivate bounds. Secondly, there was no
evidence that the extra terms due to privacy (involving $\eps$) in the
accuracy guarantee were necessary.


\subsection{Contributions}

\paragraph{New Upper Bounds for Estimating $k$-Block Graphons}
Our main focus is the problem of estimating a bounded normalized
graphon $W$ via a node-differentially private algorithm. The estimation algorithm observes one sample of a $\rho W$-random graph, and outputs the description of a graphon $\hat W$ that it hopes is close to $W$.  We consider algorithms that output a graphon with a succinct description---namely, we assume the estimate  $\hat W$ is a $k$-block graphon with equal-weight  blocks (such a graphon can  be described by a $k\times k$ symmetric matrix). The parameter $k$ offers a regularization of sorts, trading off the model's expressivity for complexity.
We measure the algorithm's error by the expected squared $\delta_2$ distance (see~\eqref{delta-def}) between $\hat W$ and $W$. 
\citet{BorgsCS15} studied this problem, developing an inefficient estimation procedure (henceforth the ``BCS'' algorithm \citep[Algorithm 1]{BorgsCS15}) and establishing an upper bound on its error. 

Our first contribution is a new analysis of the BCS algorithm that significantly improves the error bound, matching the (tight) nonprivate bounds for a large range of parameters. 
The new and old results can be summarized as the following upper bound on the mean squared error  $\E{\delta_2(\hat W, W)^2}$ of the following form.

\begin{theorem*}[Informal]
Fix some $k \geq 1$ and let $\mathcal{A}$ be the BCS algorithm. Then for all bounded graphons $W$,
\begin{align*}
  \operatorname*{\mathbb{E}}_{ G \sim G_n(\rho W)} [ \delta_2(\mathcal{A}(G),W)^2]=O\left(\underbrace{
    \begin{array}{c}
      \text{``agnostic and} \\ 
      \text{sampling errors"}
    \end{array}
 +\frac{k^2\log n}{n \epsilon}+ \frac{1}{n^2 \rho^2
     \epsilon^2}}_{\text{ as in \citet{BorgsCS15}}}+\underbrace{\frac{k-1}{n}+\left(\frac{\log k}{ \rho n}+\frac{k^2}{\rho n^2}\right)}_{\text{ improving quadratically  \citet{BorgsCS15}}}\right)
\end{align*}
\end{theorem*}

Here, the phrase ``agnostic and sampling errors'' covers two terms that are present in both bounds. The ``agnostic error'' corresponds to the distance from the true graphon $W$ to the nearest $k$-block graphon---a model misspecification error. It is unavoidable for  algorithms that output $k$-block graphons.
The ``sampling'' term corresponds to the expected distance between the true graphon $W$ and the probability matrix $(W(x_i,x_j))_{i,j=1}^{n}$ defining the $\rho W$-random graph. This distance is a random variable that can be bounded in different ways depending on what is known about $W$. If $W$ is itself a $k$-block graphon, then the the agnostic error is 0, and the sampling error (about $\sqrt{k/n}$ with high probability) is subsumed by the other error terms. 
%


Notice that our improvement to the accuracy bound lies in the ``non-private'' (that is,  independent of $\eps$) part of the error. 

This nonprivate part of our new bound is in fact optimal, as it
matches the lower bounds for \emph{nonprivate}
algorithms. Specifically, consider the case that the true graphon $W$
is in fact a $k$-block graphon and define the rate 
\begin{align*}
&R_k(\rho,\epsilon,n)=\min_{\substack{\mathcal{A} \\ \epsilon-\text{node-DP}} } \max_{ W k-\text{block }} \operatorname*{\mathbb{E}}_{ G \sim G_n(\rho W)} [ \delta_2(\mathcal{A}(G),W)^2].
\end{align*} 

 \citet[Prop. 3.4]{KloppTV15}
(and~\citet[Theorem 3]{McMillanS17}) establish the best rate if we allow any
estimation algorithm $\mathcal{A}$---private or not---to be
$$\Theta\left(\min\{ \sqrt{\frac{k}{n}}+ \left(\frac{\log k}{ \rho n}+\frac{k^2}{\rho n^2}\right),1\}\right)\, \text{for }k\geq 2.$$In particular, focusing on $\epsilon$-node-DP algorithms we conclude that for any $k \geq 2$, \begin{equation}\label{Eq1}
R_k(\rho,\epsilon,n)=\Omega\left(\min\{ \sqrt{\frac{k}{n}}+ \left(\frac{\log k}{ \rho n}+\frac{k^2}{\rho n^2}\right),1\}\right)\
\end{equation}
Notice that our upper bound as established in Theorem 1 matches exactly this lower bound when the true graphon has exactly $k$
blocks and $\epsilon$ is sufficiently large (since then the agnostic
error is 0, the sampling error is know to be $O(\sqrt{k/n})$, and the
$\eps$-dependent terms go to 0). In particular, using Theorem 1 we conclude a tight characterization of the $\epsilon$-independent part of the rate $R_k(\rho,\epsilon,n)$,
\begin{corollary*}[Informal]
Fix some $k \geq 2$. Then there exists an algorithm such that for all bounded graphons $W$,
\begin{align*}
  R_k(\rho,\epsilon,n)= O\left(
 \frac{k^2\log n}{n \epsilon}+ \frac{1}{n^2 \rho^2
     \epsilon^2}
\right)+\underbrace{O\left(\sqrt{\frac{k-1}{n}}+\left(\frac{\log k}{ \rho n}+\frac{k^2}{\rho n^2}\right)\right)}_{\text{ tight nonprivate part based on (\ref{Eq1})}}
\end{align*}
\end{corollary*}

\paragraph{Additional Error Due to Privacy $(k \geq 2)$}
To understand whether we have found
the true minimax rate, it remains to understand whether the terms
based on $\eps$ are optimal. We show that the second of these cannot be improved, on the slightly less restrictive case where the blocks of the $k$-block graphon can have different sizes, a set we denote by $\tilde{W}[k]. $
\begin{theorem*}[Informal]
For $k \geq 2$, $$\tilde{R}_k(\rho,\epsilon,n)=\Omega
\left(\frac{1}{n^2 \epsilon^2} \right),$$ where
$\tilde{R}_k(\rho,\epsilon,n)$ is defined to be
\begin{align}\label{GrRate00}
\min_{\mathcal{A} \text{ } \epsilon-\text{node-DP} } \max_{ W \in \tilde{W}[k]} \mathbb{E}_{ G \sim G_n(\rho W)} [ \delta_2(\mathcal{A}(G),W)^2].
\end{align} 
\end{theorem*}
%
%


%

This lower bound applies even to algorithms that simply estimate the
unknown density parameter $\rho$.  
The proof of this lower bound is fairly simple, relying on the fact
that even if the connection probabilities of a 2-block graphon are
known, estimating the graphon requires one to accurately estimate the
probability mass of the two blocks. We reduce to this latter problem
from the well-studied problem of estimating the bias of a sequence of
$n$ coin flips differentially privately.

We leave open the question of whether the term
$\frac{k^2 \log n }{n\epsilon}$ is necessary. 

\paragraph{The Case of Erdős-Renyi Graphs (1-Block Graphons)}
The upper bounds above all apply for $k=1$, in particular, but the
lower bounds generally do not yield anything interesting in that
case. The case of $k=1$ corresponds to graphs generated according to
the well-studied Erdős-Renyi model, where each possible edge appears
independently with an unknown probability $p$. To phrase this as an
estimation problem, consider the scale parameter $\rho$ to be known,
and the algorithm's goal is to estimate a constant graphon $W(x,y)=p$
subject to $p\leq \rho$. 
(Unlike in the case of larger $k$, estimating the
normalized graphon $W$ is trivial since, after normalization,
$W(x,y)=1$.)

Nonprivately, the optimal estimator is the edge density of the observed graph,
$(\#\text{edges}) / {\binom{n}{2}}$.

What about private algorithms? First, observe that the algorithm $\mathcal{A}_0$ that adds Laplace noise of order
$\frac{1}{n\epsilon}$ to the edge density is $\epsilon$-node
differentially private. Furthermore, for $p \in [0,\rho]$, $$\mathbb{E}_{G \sim G_{n,p}}\left[|\mathcal{A}_0(G)-p|^2\right]=O\left(\frac{\rho}{n^2}+ \frac{1}{n^2\epsilon^2}\right).$$

Potentially surprisingly, we establish that the rate obtained this way
is not optimal. As we explain in section 3, the main reason for the
suboptimality of this method is that it is based on calculating the
worst-case sensitivity of the edge density over the space of all
undirected graphs. In particular, this estimator ignores the rich
structure of the Erdos-Renyi graphs. Using this structure, we
establish a series of results relating the node-distance and the
Erdos-Renyi graphs
\ifnum\focsversion = 0
(Lemma \ref{union}, Lemma \ref{homog})
\fi
along-side with a general extension result (Proposition \ref{extension}) which combined allows to prove the following improved upper bound 

\begin{theorem*}[Informal]
There exists an $\epsilon$-node-DP algorithm $\mathcal{A}$ such that for any $\rho \in (0,1]$,
\begin{equation}
\max_{ p \in [0,\rho]} \mathbb{E}_{ G \sim G_{n,p}} [  ( A(G) - p )^2]=O\left(\frac{\rho}{n^2}+ \frac{\log n}{n^3\epsilon^2}\right).
\end{equation}

\end{theorem*}

Using the same techniques we are able to establish the corresponding
result for the uniform $G(n,m)$ model which obtains an
error $$O\left(\frac{\log n}{n^3\epsilon^2}\right),$$ avoiding the
edge-density variance term which appears in the Erdos-Renyi case,
$\frac{\rho}{n^2}$. We end this section with a novel lower bound for $G(n,m)$ model.

\begin{theorem*}[Informal]Suppose $\epsilon$ is a constant. Then,
\begin{align*}
\min_{\mathcal{A} \epsilon-\text{node-DP}} \max_{ m \in [\frac{1}{3}\binom{n}{2},\frac{2}{3}\binom{n}{2}]} &\mathbb{E}_{ G \sim G(n,m)} [  ( A(G) - \frac{m}{\binom{n}{2}} )^2]=\Omega\left(\frac{1}{n^3\epsilon^2}\right).
\end{align*}

\end{theorem*}This Theorem establishes that the upper bound for the $G(n,m)$ model is optimal up-to-logarithmic terms in the $\epsilon$-constant regime and suggests the same for the Erdos-Renyi case.

\paragraph{A General Extension Result}
In Section 4, we present in detail the general extension result we used in Section 3 as it could be of independent interest. The extension result works for an arbitrary $\epsilon$-differentially private algorithm which receives input from a metric space $(M,d)$ and outputs distributions of an arbitrary output measurable space $(\Omega,\mathcal{F})$. We establish that if there exists such an $\epsilon$-differentially private algorithm $\hat{\mathcal{A}}$ defined only on a subset of the input space $\mathcal{H}$, the algorithm can be extended to an $2\epsilon$-differentially private algorithm $\mathcal{A}$ defined on the whole input space $M$ such that if the input $G \in \mathcal{H}$, the distributions of the output of $\hat{\mathcal{A}}(G)$ coincides with the distribution of $\mathcal{A}(G)$. 

\section{Notation and Preliminaries}
\label{sec:defs}

\paragraph{$k$-block Graphons} For every $k \in \mathbb{N}$, we embed the set of $k \times k$ symmetric matrices into the space of graphons as following: let $\mathcal{P}_k=(I_1,\ldots,I_k)$ be the partition of $[0,1]$ into adjacent intervals of lengths $\frac{1}{k}$. For $A \in \mathbb{R}^{k \times k}_{ \geq 0}$ define $W[A]$ to be the step function which equals $A_{ij}$ on $I_i \times I_j$, for every $i,j \in [k]$. We say a graphon $W$ is a $k$-block graphon if $W=W[A]$ for some $A \in \mathbb{R}^{k \times k}_{ \geq 0}$ and denote by $\mathcal{W}[k]$ the space of $k$-block graphon.
\paragraph{Distances between Graphons}

For $A,B$ symmetric $n \times n$ matrices and a graphon $W$ we set for convenience $\delta_2(A,W)=\delta_2(W[A],W)$ and $\delta_2(A,B)=\delta_2(W[A],W[B])$, where $\delta_2$ is defined for two graphons in \ref{delta-def}. Furthermore we focus also on the, in general larger than $\delta_2$, distance $$\hat{\delta}_2(A,W)=\inf_{\pi \in \mathcal{S}_n}\|W[A^{\pi}]-W\|_2,$$ where $\pi$ ranges over all permurations of $\{1,2,\ldots,n\}$ and for all $i, j \in [n]$, $A^{\pi}_{ij}=A_{\pi(i),\pi(j)}$. $\delta_2$ is in principle smaller than $\hat{\delta}_2$ as it minimizes the $\ell_2$ distance over all measure-preserving transformations, while the latter distance minimizes only on such transformation that can be expressed as permutations of the rows and columns of the underlying matrix $A$.

We consider two fundamental types of errors of approximation of $W$.


The \textbf{agnostic error}, or \emph{oracle error}, of approximating $W$ by a $k$-block graphon with respect to $\delta_2$ and $\hat{\delta}_2$, $$\epsilon^{(O)}_k(W)=\min_B \delta_2(B,W)$$and $$\hat{\epsilon}^{(O)}_k(W)=\min_B \hat{\delta}_2(B,W),$$ where $B$ ranges over all matrices in $\mathbb{R}^{k \times k}$. The agnostic errors corresponds to the model mispecification errors of the statistical problem of estimating $W$ using a $k$-block graphon.  We consider them as benchmarks for our approach, and the errors an ``oracle" could obtain (hence the superscript $O$).  

\textit{Scale of agnostic error:}
For any bounded $W$, both $\epsilon^{(O)}_k(W) \text{  and } \hat{\epsilon}^{(O)}_k(W)$ tend to zero as $k \rightarrow +\infty$ (see \cite[Sec.~2]{BorgsCS15}  for details).
\ifnum\focsversion=0
Furthermore, if $W$ is $\alpha$-Holder continuous for some $\alpha \in (0,1)$,
i.e. if for some $C>0$, $|W(x,y)-W(x',y')| \leq C \delta^{\alpha}$ if $|x-x'|+|y-y'| \leq \delta$, then both $\epsilon^{(O)}_k(W) \text{  and } \hat{\epsilon}^{(O)}_k(W)$ are of order $O(k^{-\alpha})$. 
\fi

  The \textbf{sampling error} of approximating $W$ from $G=G_n(\rho W)$ with respect to $\hat{\delta_2}$, $$\epsilon_n(W)= \hat{\delta}_2(H_n(W),W).$$ Recall that the only information for $W$ in the observed graph $G$ comes from the edge probabilities $H_n(i,j)=\rho W(x_i,x_j)$ where $x_i$ are the iid uniform in $[0,1]$ labels of the vertices. Intuitively, a large discrepancy between the edge probability matrix $H_n(W)$ and $W$ results in bad estimation of $W$ given $G$. Unlike the agnostic error, the sampling error is a random variable (depending on the assignment of nodes to ``types'' in $[0,1]$.)

\textit{Scale of sampling error: }
For any bounded $W$, $\epsilon_n(W) \stackrel{P}{\longrightarrow} 0$ as $n \rightarrow + \infty$ \cite[Lemma 1]{BorgsCS15}. Furthermore, if additionally $W$ is a $k$-block graphon it can be established that  $\epsilon_n(W)=O( \sqrt[4]{\frac{k}{n}})$ with probability tending to one as $n \rightarrow + \infty$ \cite[Appendix D]{BorgsCS15}.
\ifnum\focsversion=0
Finally, if $W$ is $\alpha$-Holder continuous then $\epsilon_n(W)=O( n^{-\frac{\alpha}{2}})$ with probability tending to one as $n \rightarrow + \infty$.
\fi



\section{Private Graphon Estimation}\label{privsec}

\paragraph{Model}
Let $k, n \in \mathbb{N}$ with $k \leq n$, $\Lambda \geq 1$ and $\epsilon >0$. Suppose $W$ is an unknown normalised graphon with $\|W\|_{\infty} \leq \Lambda$. For some unknown ``sparsity level" $\rho=\rho_n \in (0,1)$ with $\rho \Lambda\leq 1$, the analyst observes a graph $G$ sampled from the $\rho W$-random graph, $G_n(\rho W)$. The analyst's goal is to use an $\epsilon$-node-DP algorithm $\mathcal{A}$ on $G$ to output a $k$-block model approximation of $W$, say $W[\hat{B}]$	for $\hat{B} \in \mathbb{R}^{k \times k}$, which minimizes the mean squared error,  $$\mathbb{E}_{ G \sim G_n(\rho W), \hat{B} \sim \mathcal{A}(G)} [ \delta_2(\hat{B},W)^2].$$

%
%
%

\subsection{Main Algorithm}

We use the same algorithm as Borgs et al. \cite{BorgsCS15}, described in Algorithm 1.

\paragraph{Notation for Algorithm 1} For $k,n\in \mathbb{N}$ with $k \leq n$, we say that $\pi: [n] \rightarrow [k]$ is a $k$-equipartition of $[n]$, if it partitions $[n]$ into $k$ classes  such that is for every $i \in [n]$, $||\pi^{-1}(i)|-\frac{n}{k}|<1$. For a matrix $Q \in \mathbb{R}^{k \times k}$ and a matrix  $A \in \mathbb{R}^{n \times n}$, we set $\mathrm{Score}(Q, \pi, A)=\|A\|_2^2-\|A-Q_{\pi}\|_2^2$, where $\pi$ ranges over all $k$-equipartitions of $[n]$,$(Q_{\pi})_{i,j}=Q_{\pi(i),\pi(j)}$ for all $i,j \in [n]$ and $ \|A\|_2=\left(\frac{1}{n^{2}}\sum_{i,j=1}^n A^2_{ij} \right)^{\frac{1}{2}}$. Finally, we denote by $\mathcal{G}_n$ the space of undirected graphs on $n$ vertices and $\mathcal{G}_{n,d}$ the subset of graphs in $\mathcal{G}_n$ where the maximum degree is bounded by $d$. 

We now describe the steps of the algorithm. The algorithm takes as input
the privacy parameter $\eps$, the graph $G$, a number $k$ of blocks, and a constant $\lambda\geq 1$
that will have to be chosen large enough to guarantee consistency of the algorithm.

\begin{algorithm2e}\caption{Private Estimation Algorithm}
\label{alg:main-algo}
\KwIn{$\eps>0$, $\lambda\geq 1$, an integer $k$ and graph $G$ on $n$
  vertices.}
\KwOut{$k$-block graphon (represented as a $k\times k$ matrix $\hat B$)
estimating $\rho W$}

\label{step:rho-approx}Compute  an $(\eps/2)$-node-private density approximation $\hat{\rho}=e(G)+\mathrm{Lap}(4/n\eps)$ \;

$d=\lambda \hat{\rho} n$ (the target maximum degree) \;
$\mu = \lambda \hat{\rho}$ (the target
  $L_\infty$ norm for $\hat{B}$) \;

For each $B$ and $\pi$, let $\widehat{\mathrm{Score}}(B,\pi;\cdot)$ denote a
          nondecreasing Lipschitz extension (from \cite{KNRS13}) of $\mathrm{Score}(B,\pi;\cdot)$ from
          $\mathcal{G}_{n,d}$ to $\mathcal{G}_n$ such that for all matrices $A$,
          $\widehat{\mathrm{Score}}(B,\pi;A)\leq \mathrm{score}(B,\pi;A)$, and define
          \[\widehat{\mathrm{Score}}(B;A) = \max_{\pi} \widehat{\mathrm{Score}}(B,\pi;A)\]

\Return $\hat{B}$, sampled from the distribution
\[
\Pr(\hat{B} = B)\propto\exp\left(\frac \eps {4\Delta} \widehat{\mathrm{Score}}(B;A)\right),
\]
where $\displaystyle \Delta=\frac{4d \mu}{n^2}=\frac{4\lambda^2\hat{\rho}^2}{n}$
and $B$ ranges over matrices in 
$$ \textstyle
\B_{\mu}= \{B\in [0,\mu]^{k\times k}: \text{all
          }B_{i,j}\text{ are multiples of }\frac 1 n\};
$$
\end{algorithm2e}

\paragraph{Main Result}
Algorithm 1 is proven to be $\epsilon$-node-DP at \cite[Lemma 3]{BorgsCS15}. Borgs et al gave upper bound on its worst-case mean squared error, $\mathbb{E}_{ G \sim G_n(\rho W), \hat{B} \sim \mathcal{A}_G} [ \delta_2(\hat{B},W)^2]$. We state the improved bound here:
\begin{theorem}\label{stronger}
Suppose
 \begin{itemize}
\item  $\frac{6 \log n}{n}<\rho \leq \frac{1}{\Lambda}$, $8 \Lambda \leq \lambda  $, \text{ and }
\item $\rho n \epsilon/ \log n \rightarrow +\infty$, $\epsilon=O(k^2 \log n/\lambda^3)$
\end{itemize} 
Then the $\epsilon$-node-DP Algorithm 1 from \cite{BorgsCS15}, $\mathcal{A}$, with input $\epsilon,\lambda,k$ and $G$ outputs a pair $(\hat{\rho},\hat{B}) \in [0,1] \times [0,1]^{k \times k}$ with  $\mathbb{E}_{ G \sim G_n(\rho W), \hat{B} \sim \mathcal{A}_G} [ \delta_2(\frac{1}{\hat{\rho}}\hat{B},W)^2] $ of the order
\begin{align*}
&O\left(\mathbb{E}\left[\epsilon_k^{(O)}(W)^2\right]+\mathbb{E}\left[\epsilon_n(W)^2\right]+\lambda  \frac{k-1}{n}\right)+O \left(\lambda  \left(\frac{\log k}{ \rho n}+\frac{k^2}{\rho n^2}\right) + \lambda^2 \frac{k^2\log n}{n \epsilon}+ \frac{\lambda^2}{n^2 \rho^2 \epsilon^2}\right).
\end{align*}
\end{theorem}


The bound from Theorem 1 in \cite{BorgsCS15} states that, under slightly different parameter assumptions, the mean squared error $\mathbb{E}_{ G \sim G_n(\rho W), \hat{B} \sim \mathcal{A}_G} [ \delta_2(\frac{1}{\hat{\rho}}\hat{B},W)^2]$ is at most \begin{align*}
 &O\left(\mathbb{E}\left[\epsilon_k^{(O)}(W)^2\right]+\mathbb{E}\left[\epsilon_n(W)^2\right]\right)+O\left(\sqrt{\lambda^2  \left(\frac{\log k}{ \rho n}+\frac{k^2}{\rho n^2}\right)}\right)+O\left( \lambda^2\frac{k^2\log n}{n \epsilon}+ \frac{\lambda^2}{n^2 \rho^2 \epsilon^2}\right).
\end{align*}The improvement therefore of our result lies on the $\epsilon$-independent part of the bound. For convenience, we call this part of the bound  the non-private part of the bound and the $\epsilon$-dependent part, the private part of the bound. As we establish in the following subsection, the improvement of Theorem \ref{stronger} on the non-private part is the optimal possible.

\paragraph{The $k$-block Estimation Rate}In this subsection we focus on the case $W$ is a $k$-block graphon and establish that the improvement of Theorem 1 on the non-private part of the bound is optimal in the following sense. For some $k \geq 1$, assume that $W \in \mathcal{W}[k]$ with $\|W\|_{\infty} \leq \Lambda$, that is $W=W[B]$ for some $B \in [0,\Lambda]^{k \times k}$. Restricting ourselves to the specified subset of graphons we consider the minimax rate,\begin{align*}\label{GrRate}
\min_{\mathcal{A} \text{ } \epsilon-\text{node-DP} } \max_{ W \in \mathcal{W}[k],\|W\|_{\infty} \leq \Lambda} \mathbb{E}_{ G \sim G_n(\rho W)} [ \delta_2(\mathcal{A}_G,W)^2].
\end{align*}which we denote by $R_k(\rho,\epsilon,\Lambda,n)$.

 If $k \geq 2$, Theorem 3 from~\cite{McMillanS17} and (up-to-$\log k$ factors) Proposition 3.4 of~\cite{KloppTV15}, establishes that this rate, under no differential-privacy constraint (a case corresponding to $\epsilon$ ``equal to" $+\infty$ for our purposes), behaves like $$\Theta\left(\min\{\Lambda^2 \sqrt{\frac{k}{n}}+\Lambda  \left(\frac{\log k}{ \rho n}+\frac{k^2}{\rho n^2}\right),\Lambda^2\}\right).$$ This result does not directly apply to our setting as we consider only finite $\epsilon>0$. Note, though, that $\epsilon$-node-DP is an increasing property, as if an algorithm is $\epsilon$-node-DP, it is also $\epsilon'$-node-DP for any $\epsilon'>\epsilon.$ Therefore $R_k(\rho,\epsilon,\Lambda,n)$ is a non-increasing function of $\epsilon$, as increasing $\epsilon$ only can shrink the feasible sets of estimators. Hence, the result from ~\cite{KloppTV15} provides a lower bound for the rate $R_k(\rho,\epsilon,\Lambda,n)$. Combined with Theorem \ref{stronger} we obtain a tight characterization of the non-private part of the rate $R_k(\rho,\epsilon,\Lambda,n)$, and establish that Algorithm 1 from \cite{BorgsCS15} obtains the optimal non-private part of the rate.
\begin{coro}\label{coro1}Suppose $k \geq 2$. Under the assumptions of Theorem \ref{stronger} and the additional assumption $\rho n \geq k-2$,
\begin{align*}
 &R_k(\rho,\epsilon,\Lambda,n)=\Omega\left(\min\{\Lambda^2 \sqrt{\frac{k}{n}}+\Lambda  \left(\frac{\log k}{ \rho n}+\frac{k^2}{\rho n^2}\right),\Lambda^2\}\right) \end{align*}
and
\begin{align*}
&R_k(\rho,\epsilon,\Lambda,n) =  O\left(\Lambda^2 \sqrt{\frac{k}{n}}+\Lambda  \left(\frac{\log k}{ \rho n}+\frac{k^2}{\rho n^2}\right) \right)+ O \left(\Lambda^2  \frac{k^2\log n}{n \epsilon}+ \frac{\Lambda^2}{n^2 \rho^2 \epsilon^2}\right),
\end{align*}where the upper bound is achieved by Algorithm 1 from \cite{BorgsCS15}. 
\end{coro}


%

\paragraph{A Lower Bound on the Private Part}
In this subsection we establish for $k \geq 2$ a lower bound on the private part of the rate. We establish that the term of order $\frac{\Lambda^2}{n^2 \rho^2 \epsilon^2}$ appearing in the upper bound of Theorem 1 is necessary, up to the dependence on $\rho,\Lambda$. For the lower bound we focus on $k$-block graphons $W=W[B]$ with potentially slightly-unequal sizes, we do not require them to be normalized, and we set $\rho=\Lambda=1$. Specifically, let $\tilde{\mathcal{W}}[k]$ be the set of all graphons $W$ for which $\|W\|_{\infty} \leq 1$ and for some $A \in \mathbb{R}^{k \times k}_{ \geq 0}$ and some $\mathcal{P}_k=(I_1,\ldots,I_k)$ partition of $[0,1]$ into adjacent intervals of (potentially different) lengths in $[\frac{1}{4k},\frac{4}{k}]$, $W$ is the step function which equals $A_{ij}$ on $I_i \times I_j$, for every $i,j \in [k]$. Let also \begin{align*}
\tilde{R}_k(\epsilon,n)=\min_{\mathcal{A} \text{ } \epsilon-\text{node-DP} } \max_{ W \in \tilde{\mathcal{W}}[k]} \mathbb{E}_{ G \sim G_n( W), \hat{B} \sim \mathcal{A}_G} [ \delta_2(\hat{B},W)^2].
\end{align*}

\begin{theorem}\label{kbiglower}
Suppose $k \geq 2$. Then 
\begin{equation*}
\tilde{R}_k(\epsilon,n) =\Omega \left( \frac{1}{n^2 \epsilon^2} \right).
\end{equation*}
\end{theorem}


\section{Private Estimation of Erdos Renyi Graphs (1-Block Graphons)}

This section is devoted to the study of the privately estimating $k$-block graphons in the special case $k=1$. Since for $k=1$ the graphon corresponds to a constant function, we deal with the fundamental question of estimating privately the parameter of an Erdos-Renyi random graph model. Note that since the graphon is constant, to make the estimation task non-trivial we do not adopt the assumption that the graphon is normalized.  Furthermore, using the notation of the previous section for reasons of simplicity we focus on the case $\rho$ is known to the analyst and $\Lambda=1$. 

Using such a graphon $W $, we conclude that for some $p_0 \in [0,1]$ $W(x,y)=p_0$ for every $x,y \in [0,1]$ and the analyst's observes simply a sample from an Erdos Renyi random graph with $n$ vertices and parameter $p:=\rho \cdot p_0 \leq \rho$. Multiplying the rate by the known $\rho$, the goal becomes to estimate $p$ using an $\epsilon$-differentially private algorithm. In agreement with the non-private behavior where the estimation rate is provably much smaller when $k=1$ compared to $k>1$ (see Sec. 3.2 in \cite{KloppTV15} for details), we reveal a similar behavior in the case of private estimation. In particular, based on Theorem \ref{kbiglower} for $k>1$ and $\Lambda=1$ the rate of interest is $$\Omega \left( \frac{1}{n^2 \epsilon^2} \right).$$ Here we establish that the $\epsilon$-dependent part of the rate for $k=1$ drops to $$O\left(\frac{\log n}{n^{3}\epsilon^2}\right).$$

\subsection{A New Algorithm for Density Estimation in Erdos Renyi Random Graphs}
\label{density}

The rate we want to find is for $\rho \in [0,1]$, $$R(\rho,\epsilon,n)=\min_{\mathcal{A} \text{ } \epsilon-\text{node-DP} } \max_{ p \in [0,\rho]} \mathbb{E}_{ G \sim G_{n,p}} [  ( A(G) - p )^2].$$

A standard $\epsilon$-node-DP algorithm for this task is the addition of appropriate Laplace noise to the edge density of the graph $G$ (Lemma 10 of \cite{BorgsCS15}). The global sensitivity (Definition 2 in \cite{BorgsCS15}) of the edge density with respect to the node-distance can be easily proven to be of the order $\Theta(\frac{1}{n})$. In particular it is upper bounded by $\frac{4}{n}$, as if $G,G' \in \mathbb{G}_n$, $$|e(G)-e(G')| \leq \frac{4}{n}d_v(G,G').$$Therefore, using Lemma 10 of \cite{BorgsCS15}, the addition of $\mathrm{Lap}(\frac{4}{n\epsilon})$ noise to the edge density provides an $\epsilon$-node-DP estimator.  This estimator allows us to conclude the following Lemma.
\begin{lemma}\label{Lap} For any $\rho ,\epsilon >0$, 
$$R(\rho,\epsilon,n)=O\left(\frac{\rho}{n^2}+ \frac{1}{n^2\epsilon^2}\right).$$

\end{lemma}


As we establish in Theorem \ref{thm32} the upper bound of Lemma \ref{Lap} is, potentially surprisingly, not tight. A weakness of the proposed algorithm is that it computes an estimator based on the global sensitivity of the edge density over all pairs of undirected graphs of $n$ vertices and on the other hand applies it only to graphs coming from Erdos-Renyi models. To reveal more the potential weakness of the estimator, let us consider a pair of node-neighbors $G,G'$, that is $d_v(G,G')=1$, where the difference $e(G)-e(G')$ is of the order $\frac{1}{n}$. It is easy to check that the difference can become of this order only if the degree of the rewired vertex had $o(n)$ degree in $G$ and $\Theta(n)$ degree in $G'$ or vice versa. Since the degree of every other vertex changes by at most 1, the rewired vertex in $G$ or $G'$ has either very high degree or very low degree compared to the average degree in $G$ or $G'$. Such a non-homogenuous property of the degree distribution appears, though, only with a negligible probability under any Erdos-Renyi model. This line of thought suggests that there could possibly be some ``homogeneity" set, $\mathcal{H}$, for which any graph sampled from Erdos Renyi model belongs to with probability $1-o(1)$ and the sensitivity of the edge density on pairs of graphs from $\mathcal{H}$ is much lower than $\frac{1}{n}$.

Unfortunately the existence of such a set can be proven to be non-true for the following reason. The empty graph $G_0$ (which appears almost surely for the Erdos Renyi graph with $p=0$) and the complete graph $G_1$ (which appears almost surely for the Erdos Renyi random graph with $p=1$) should be included in such ``homogeneity"  set and furthermore $$\frac{e(G_1)-e(G_0)}{d_v(G_0,G_1)}=\frac{1}{n-1}=\Theta(\frac{1}{n}).$$ 

We establish, though, that this is essentially the only ``extreme" case and such an ``homogeneity"  set $\mathcal{H}$ exists, in the following sense. There exist a set $\mathcal{H}$ which contains any Erdos Renyi graph with probability $1-o(1)$, that is $$\min_{p \in [0,1]} \mathbb{P}_{G \sim G_{n,p}}\left(G \in \mathcal{H}\right)=1-o(1),$$ and furthermore from any $G,G' \in \mathcal{H}$ either $$d_v(G,G')>n/4$$ or $$\frac{|e(G)-e(G')|}{d_v(G,G')}=O(\frac{\sqrt{\log n}}{n^{3/2}}).$$ 

This $\sqrt{n}$-improvement on the edge density sensitivity on $\mathcal{H}$ allows us to establish the existence of an $\epsilon/2$-node-DP algorithm which is defined on graphs in $\mathcal{H}$ and has mean squared error of the order $O(\frac{\log n}{n^{3}\epsilon^2})$.
Notice that the order is much lower than the performance of the addition of Laplace noise (Lemma \ref{Lap}). Next we establish a general extension result (Theorem \ref{extension}) which allows us to extend the $\epsilon/2$-node-DP algorithm defined on $\mathcal{H}$ to an $\epsilon$-node-DP on the whole space of undirected graphs with $n$ nodes. The extension has the crucial property that it outputs the same probability distributions with the original algorithm when the input belongs in $\mathcal{H}$. The extension result applies generally to any $\epsilon$-differentially private algorithm which takes values in an arbitrary metric space and outputs probability distributions of any measurable space. Since such a result could be of independent interest we devote Section \ref{Ext} solely for its presentation.

Using the extented algorithm we establish the following results for graphs sampled from the Erdos Renyi random graph $G_{n,p}$ and the uniform graph $G(n,m)$. Notice that for the $G_{n,p}$ model there exists an additional non-private term $\frac{\rho}{n^2}$. This appears only in the Erdos-Renyi case and not in the uniform model as it comes from the vanishing but non-zero variance term of the edge density in the Erdos Renyi model. 

\begin{proposition}[The $G(n,m)$ case] \label{prop:32-fixed-m} Let $\epsilon, \rho \in (0,1)$ be
  functions of $n$ such that $\epsilon n/\log n \rightarrow +\infty$. There
  is an $\epsilon$-node-DP algorithm $A$ such that, for all $m < \rho
\binom n 2$,

$$\bE_{G\sim G(n,m)} \Big| A(G) - \frac{m}{\binom{n}{2}} \Big|^2 =  O \left(\max\left\{\rho,\tfrac{\log n}{n}\right\}  \cdot\frac{\log n}{n^{3} \epsilon^2} \right).$$
\end{proposition}

\begin{theorem}[The Erdos-Renyi case]\label{thm32}
If $\epsilon \in (0,1)$ with $\epsilon n/\log n \rightarrow +\infty$, then $$R(\rho,\epsilon,n) = O \left(
  \frac{\rho}{n^2}+\max\{\rho,\frac{\log n}{n}\} \frac{\log n}{n^{3} \epsilon^2} \right).$$
\end{theorem}

\subsection{Lower bounds for $G(n,m)$} \label{sec:lower-one-block}

In this subsection we dicuss the complementary question of lower bounds for the edge density estimation question in random graphs. We establish that when $\epsilon$ in constant and the graph is generated by the uniform model $G(n,m)$, the bound implied by Proposition~\ref{prop:32-fixed-m} is tight.

We establish this by first proving the following proposition on coupling of $G(n,m)$ models with varying $m$ which could be of independent interest.
\begin{proposition}\label{coupl}
  Let $n$ be sufficiently large, and $k$ an arbitrary function of $n$ which is $o(\sqrt{n})$. Let $m = \frac{1}{2}{\binom n 2} - \frac k 2$ Let $P = G(n,m)$ and
  $Q = G(n,m +k )$.  There exists a coupling of $(G,H)$ of $P$ and $Q$
  such that, with probability tending to one, one can obtain $G$ from
  $H$ by rewiring one vertex.
\end{proposition}

Using the proposition we establish the following lower bound.
\begin{theorem}\label{thmcoupl}
  Let $\epsilon>0$ be a constant positive number, $n \in \mathbb{N}$, $m = \frac{1}{2} \binom n 2
  - \tfrac k 2$ and $k$ an arbitrary function of $n$ which is $o(\sqrt{n})$ . Then there exists a $\beta=\beta(\epsilon)\in (0,1)$ such that no
  $\epsilon$-node DP private algorithm can distinguish $G(n,m)$ from
  $G(n,m +k )$ with probability bigger than $\beta(\epsilon)>0$.  In particular, the upper bound of
  Proposition~\ref{prop:32-fixed-m} is tight up-to-logarithmic terms for constant $\epsilon$
  and $\rho$. 
\end{theorem}


\section{A General Extension Technique}\label{Ext}
In this section we describe the general extension technique which allowed us to conclude the upper bound in Theorem \ref{thm32}. Since the technique applies generally to the extension of any $\epsilon$-differentially private algorithm from any input metric space to any output measurable space, we present it here for the following general model.

\paragraph{The Model}
Let $n \in \mathbb{N}$ and $\epsilon>0$. We assume that the analyst's objective is to estimate a certain quantity which takes values in some measurable space $(\Omega,\mathcal{F})$ from input data which take values in a metric space $(\mathcal{M},d)$. The analyst is assumed to use for this task a randomized algorithm $\mathcal{A}$ which should be \begin{itemize}
\item[(1)] as highly \textbf{accurate} as possible for input data belonging in some \textit{hypothesis set} $\mathcal{H} \subseteq \mathcal{M}$;

\item[(2)] $\epsilon$-\textbf{differentially private} on the whole metric space of input data $(M,d)$.
\end{itemize}

In this section we state the following result. Consider an arbitrary $\epsilon$-differentially private algorithm defined on input belonging in some set $\mathcal{H} \subset \mathcal{M}.$ We show that it \textbf{can be always extended} to a $2\epsilon$-differentially private algorithm defined for arbitrary input data from $\mathcal{M}$ with the property that if the input data belongs in $\mathcal{H}$, the distribution of output values is the same with the original algorithm. We state formally the result. 

\begin{proposition}[``Extending Private Algorithms at $\epsilon$-cost"]\label{extension}
Let $\hat{\mathcal{A}}$ be an $\epsilon$-differentially private algorithm designed for input from $\mathcal{H} \subseteq \mathcal{M}$. Then there exists a randomized algorithm $\mathcal{A}$ defined on the whole input space $\mathcal{M}$ which is $2\epsilon$-differentially private and satisfies that for every $D \in \mathcal{H}$, $\mathcal{A}(D) \overset{d}{=}  \hat{\mathcal{A}}(D)$.
\end{proposition}

\ifnum\focsversion = 0
\newpage

\section{Proofs for Section 2}\label{pfsec2}

\paragraph{Definitions and Notation}
For $A,B \in \mathbb{R}^{n \times n}$ and $1 \leq p < \infty$ we use the normalised $p$-norms, $$\|A\|_p=\left(\frac{1}{n^{2}}\sum_{i,j=1}^n A^p_{ij} \right)^{\frac{1}{p}}$$and the normalised inner product $$\inner{A}{B}=\frac{1}{n^2}\sum_{i,j=1}^{n}A_{ij}B_{ij}.$$

For $k,n\in \mathbb{N}$ with $k \leq n$, we say that $\pi: [n] \rightarrow [k]$ is a $k$-equipartition of $[n]$, if it partitions $[n]$ into $k$ classes all of which have size as close to $\frac{n}{k}$ as possible, that is for every $i \in [n]$, $||\pi^{-1}(i)|-\frac{n}{k}|<1$. For $Q \in \mathbb{R}^{k \times k}$ and $\pi$ a $k$-equipartition of $[n]$ we set $Q_{\pi}$ the block matrix given by $(Q_{\pi})_{ij}=Q_{\pi(i)\pi(j)}$. Furthermore, if $\pi$ a $k$-equipartition of $[n]$ and $B \in \mathbb{R}^{n \times n}$ we let $B(\pi) \in \mathbb{R}^{k \times k}$ be the matrix with entries $$B(\pi)_{ij}=\frac{1}{|\pi^{-1}(i)||\pi^{-1}(j)|} \sum_{ l \in \pi^{-1}(i), m \in \pi^{-1}(j)} B_{l,m}.$$Notice that clearly $B(\pi)=\mathrm{argmin}_{B \in \mathbb{R}^{k \times k}} \|B_{\pi}-B\|_2$ and $\|B(\pi)_{\pi}-B\|_2=\min_{B \in \mathbb{R}^{k \times k}} \|B_{\pi}-B\|$.

 We define the score function defined in \cite{BorgsCS15}; for a matrix $Q \in \mathbb{R}^{k \times k}$ and a matrix  $A \in \mathbb{R}^{n \times n}$ we set $\mathrm{Score}(Q,A)=\max_{\pi} \left(\|A\|_2^2-\|A-Q_{\pi}\|_2^2\right)$, where $\pi$ ranges over all $k$-equipartitions of $[n]$. For $r \in [0,1]$ and $k,n \in \mathbb{N}$ we set $$\mathcal{B}_r=\{B \in [0,r]^{k \times k} | nB_{ij} \in \mathbb{Z}, \text{ for all } i,j \in [k] \}.$$For every  $B \in [0,r]^{n \times n}$ and $\pi$ $k$-equipartition of $[n]$, let $B(\pi,n) \in \mathcal{B}_r$ be the matrix with entries $$B(n,\pi)_{ij}=\frac{\floor{n(B(\pi))_{ij} }}{n}.$$

For some symmetric $Q \in [0,1]^{k \times k}$ such that $Q_{ii}=0$ for all $i=1,2,\ldots,n$ we say that a symmetric matrix $A \in \mathbb{R}^{k \times k}$ is distributed according to $\mathrm{Bern}_0(Q)$ if for all $i \leq j$, $A_{i,j}$ follows an independent Bernoulli with parameter $Q_{i,j}$.

\paragraph{Key Lemmata}
For this subsection we define between two matrices $B_1 \in \mathbb{R}^{k \times k}$ and $B_2 \in \mathbb{R}^{n \times n}$ the distance $$\hat{\delta}_2(B_1,B_2)=\min_{\pi}\|(B_1)_{\pi}-B_2\|_2,$$where $\pi: [k] \rightarrow [n]$ ranges over all equipartitions of $[n]$ into $k$ classes. Furthermore for $Q \in \mathbb{R}^{n \times n}$ let $\hat{\epsilon}_k^{(O)}(Q)=\min_{\pi,B \in \mathbb{R}^{k \times k} }\|B_{\pi}-Q\|_2=\min_{\pi}\|Q(\pi)_{\pi}-Q\|_2.$

Now, the essential improvement on the analysis of \cite{BorgsCS15} is coming from improving Proposition 1 of \cite{BorgsCS15} to the following proposition.
\begin{proposition}\label{prop}
Let $r \in [\frac{1}{n},1]$, $Q \in [0,r]^{n \times n}$ be a symmetric matrix with vanishing diagonal and $A \sim \mathrm{Bern_0}(Q)$. For $\hat{B} \in \mathcal{B}_r$ and $\nu \geq 0$ conditional on an event such that
\begin{equation}\label{eq:condition}
\mathcal{E} \subseteq \{\mathrm{Score}\left(\hat{B},A\right) \geq \max_{B \in \mathcal{B}_r}[\mathrm{Score}\left(B,A\right)]-\nu^2\}
\end{equation}the following holds
\begin{equation}\label{eq:mainpr1}
\mathbb{E}\left[\hat{\delta}_2\left(\hat{B},Q\right)^2|\mathcal{E}\right] \leq O\left(\hat{\epsilon}_k^{(O)}(Q)^2+\nu^2+r \left(\frac{\log k}{n}+\frac{k^2}{n^2}\right)  \right).
\end{equation}
\end{proposition}
 To prove Proposition \ref{prop} we need first two lemmata which can be established almost compeltely by the proof techniques of \cite{KloppTV15} . Despite that we fully establish them here for reader's convenience, using only one lemma from \cite{KloppTV15}.

 \begin{lemma}\label{lemma11} Under the assumptions of Proposition (\ref{prop}) for every $\pi$ $k$-equipartition of $[n]$,

\begin{equation}\label{eq:24}
\mathbb{E}\left[ \sup_{\pi} \left( \inner{Q-Q(\pi)_{\pi}}{Q-A}-\frac{1}{16}\|Q(\pi)_{\pi}-Q\|_2^2\right)\right] \leq O\left( r \frac{\log k}{n}\right)
\end{equation}
and
\begin{equation}\label{eq:2444}
\mathbb{E}\left[ \sup_{\pi} \left( \inner{Q(n,\pi)_{\pi}-Q}{Q-A}-\frac{1}{16}\|Q(n,\pi)_{\pi}-Q\|_2^2\right)\right] \leq O\left( r \frac{\log k}{n}\right)
\end{equation}
\end{lemma}
\begin{proof} We establish only (\ref{eq:24}) as (\ref{eq:2444}) follows similarly.

Recall first Bernstein's inequality which we state for reader's convenience. Let $X_1,\ldots,X_N$ independent zero-mean random variables. Suppose $|X_i| \leq M$ almost surely for all $i=1,2,\ldots,n$. Then for any $t>0$, $$ \mathbb{P} \left[ \sum_{i=1}^N X_i \geq \sqrt{ 2t \sum_{i=1}^N \mathbb{E}[X_i^2]} + \frac{2M}{3} t\right] \leq \exp\left(-t\right).$$ 

Now notice that for any $\pi$, $$\inner{Q-Q(\pi)_{\pi}}{Q-A}=\frac{2}{n^2} \sum_{i<j} (Q-Q(\pi)_{\pi})_{ij} (Q_{ij}-A_{ij}).$$Furthermore for each $i<j$, $|(Q-Q(\pi)_{\pi})_{ij} | \leq r $ and $\mathrm{Var}(A_{i,j})=\mathrm{Var}(Q_{i,j}-A_{i,j}) \leq \|Q\|_{\infty} \leq r.$ Therefore by Bernstein we conclude that for any $\pi$
$$\mathbb{P}\left[  \inner{Q-Q(\pi)_{\pi}}{Q-A} \geq \frac{2}{n} \|Q-Q(\pi)_{\pi}\|_2 \sqrt{ rt}+\frac{4}{3n^2} rt \right) \leq \exp(-t),$$for any $t>0$. Now taking a union bound over all $\pi$, which are at most $k^n=\exp( k \log n)$, we conclude that $$\mathbb{P}\left[ \exists \pi : \inner{Q-Q(\pi)_{\pi}}{Q-A} \geq \frac{2}{n} \|Q-Q(\pi)_{\pi}\|_2 \sqrt{ r(t+k \log n)}+\frac{4}{3 n^2} r(t+k \log n) \right) \leq \exp(-t),$$for any $t>0$. Using the elementary $2uv \leq u^2+v^2$ we conclude $$\mathbb{P}\left[ \exists \pi : \inner{Q-Q(\pi)_{\pi}}{Q-A}-\frac{1}{16} \|Q-Q(\pi)_{\pi}\|_2^2 \geq \frac{52}{3 n^2} r(t+k \log n) \right) \leq \exp(-t),$$for any $t>0$, or for $C=\frac{52}{3}$,
$$\mathbb{P}\left[ \sup_{ \pi} \left( \inner{Q-Q(\pi)_{\pi}}{Q-A}-\frac{1}{16} \|Q-Q(\pi)_{\pi}\|_2^2 \right) \geq \frac{C}{ n^2} r(t+n \log k) \right) \leq \exp(-t),$$for any $t>0$. Integration with respect to $t$ implies the statement of the Lemma.  \end{proof}
\begin{lemma}\label{lemma12}  Under the assumptions of Proposition (\ref{prop}),
\begin{equation}\label{eq:25}
\mathbb{E}\left[\sup_{C, \pi} \left( \inner{C_{\pi}}{Q-A}-\frac{1}{16}\|C\|_2^2 \right)\right] \leq O \left( r( \frac{\log k}{n}+\frac{k^2}{n^2}  )\right) 
\end{equation}where the optimization is over all $C \in \mathcal{B}_r$ and $\pi$ $k$-equipartitions of $[n]$.
\end{lemma}
\begin{proof}  We care to control the quantity $\sup_{C, \pi} \left( \inner{C_{\pi}}{Q-A}-\frac{1}{16}\|C\|_2^2 \right).$ The quantity of interest equals \begin{align*} \sup_{R \geq 0, \pi} &\sup_{C \in \mathcal{B}_r, \|C\|_2= R} \left( \inner{C_{\pi}}{Q-A}-\frac{1}{16}R^2 \right) \leq  \sup_{R \geq 0, \pi}\left( \sup_{C \in [0,2r]^{k \times k}:  \|C\|_2 \leq R} \left( \inner{C_{\pi}}{Q-A}\right)-\frac{1}{16}R^2 \right).\end{align*}Now fix $\pi,R$, We set $A(\pi,R)=\{C_{\pi}: C \in [0,2r]^{k \times k}:   \|C\|_2 \leq R\}$ and consider $D(R)_{\pi} \in A(\pi,R)$ with $$\inner{D(R)_{\pi}}{Q-A}=\sup_{C \in A(\pi,R)} \inner{C_{\pi}}{Q-A}.$$ What we care to bound is  $$\sup_{R \geq 0 , \pi}\left( \inner{D(R)_{\pi}}{Q-A}-\frac{1}{16}R^2 \right).$$ We split two cases. If $\|D(R)_{\pi} \|_2 \leq \frac{2r}{n}$ we have by Cauchy-Schwarz that  $$\inner{D(R)_{\pi}}{Q-A} \leq \|D(R)_{\pi}\|_2 \|Q-A\|_2\leq \frac{2r}{n},$$ since $\|Q-A\|_{\infty} \leq 1$, and therefore 
\begin{equation} \label{firstt}
\sup_{R \geq 0, \pi}\left( \inner{D(R)_{\pi}}{Q-A}-\frac{1}{16}R^2 \right) \leq \frac{2r}{n}. 
 \end{equation} 

  If $\|D(R)_{ \pi} \|_2 \geq \frac{2r}{n}$ we use Lemma 4.1 of \cite{KloppTV15}. In that lemma the authors construct a subset $ \mathcal{C}^*(\pi) \subset \{C_{\pi}: C \in [0,2r]^{k \times k}\}$ with the following two properties, 

\begin{itemize} 
\item[(1)]  $\log |\mathcal{C}^*(\pi) |=O(k^2+\log \log n)$
\item[(2)]  For any $R >0$, if $\|D(R)_{ \pi} \|_2 \geq \frac{2r}{n}$ then there exists $V_{\pi} \in \mathcal{C}^*(\pi)$ such that $\|D(R)-V\|_2 \leq \frac{\|D(R)\|_2}{4}$ and $\|D(R)-V\|_{\infty} \leq r$.
\end{itemize} Hence if $\|D(R)_{ \pi} \|_2 \geq \frac{2r}{n}$ then for the $V$ satisfying (2) we have the following two properties. First $\|V\|_2 \leq \frac{5}{4} \|D(R)\|_2<2R$. Furthermore clearly $2(D(R)_{\pi}-V_{\pi}) \in A_{\pi,R}.$ Therefore by the definition of $D$ and the case we consider we have $\inner{2(D(R)_{\pi}-V_{\pi}) }{Q-A} \leq \inner{D(R)_{\pi}}{Q-A}$ or $$\inner{D(R)_{\pi} }{Q-A} \leq 2\inner{V_{\pi}}{Q-A}.$$We conclude combining the above that$$\inner{D(R)_{\pi} }{Q-A}  \leq 2 \sup_{C \in \mathcal{C}^*(\pi)\cap \{ C: \|C\|_2 \leq 2R\}} \inner{C_{\pi}}{Q-A} .$$

Therefore it suffices to bound the expectation of the quantity $$ \sup_{R \geq 0, \pi}\left(2 \sup_{C \in \mathcal{C}^*(\pi) \cap \{ C: \|C\|_2 \leq 2R\}} \left( \inner{C_{\pi}}{Q-A}\right)-\frac{1}{16}R^2 \right)$$which is at most $$2 \sup_{R \geq 0, \pi,C \in \mathcal{C}^*(\pi)\cap \{ C: \|C\|_2 \leq 2R\}} \left( \inner{C_{\pi}}{Q-A}-\frac{1}{128}\|C\|_2^2 \right)$$which equals $$2 \sup_{ \pi,C \in \mathcal{C}^*(\pi)} \left( \inner{C_{\pi}}{Q-A}-\frac{1}{128}\|C\|_2^2 \right)$$ Combining both cases we conclude that in complete generality
\begin{equation}\label{comb} \sup_{R >0, \pi}\left(2 \sup_{C \in \mathcal{C}^*(\pi)\cap \{C: \|C\|_2 \leq R\}} \left( \inner{C_{\pi}}{Q-A}\right)-\frac{1}{16}R^2 \right) \leq \frac{2r}{n}+2 \sup_{ \pi,C \in \mathcal{C}^*(\pi)} \left( \inner{C_{\pi}}{Q-A}-\frac{1}{128}\|C\|_2^2 \right).\end{equation}

We focus on the second term. As in the proof of Lemma \ref{lemma11} we have that for any $C \in [0,2r]^{k \times k}$, $\pi$ k-equipartition of $[n]$, $$\mathbb{P}\left[  \inner{C_{\pi}}{Q-A} \geq \frac{2}{n} \|C_{\pi}\|_2 \sqrt{ 2rt}+\frac{8}{3n^2} rt \right) \leq \exp(-t),$$for any $t>0$. Using that $\|C_{\pi}\|_2=\|C\|_2$ and the elementary $2uv \leq u^2+v^2$ we conclude that for some $c_0>0$ and any $C \in [0,2r]^{k \times k}$, $\pi$ k-equipartition of $[n]$, $$\mathbb{P}\left[  \inner{C_{\pi}}{Q-A} -\frac{1}{128}\|C\|_2^2 \geq \frac{c_0}{n^2} rt \right) \leq \exp(-t),$$for any $t>0$. Hence using a union bound over all partitions $\pi$, which are at most $k^n$ and $C \in \mathcal{C}^*(\pi)$ which based on property (1) are $\exp(O(k^2+ \log \log n))=\exp( O(k^2+n \log k))$ we derive that for some $c_1>0$
$$\mathbb{P}\left[ \sup_{ \pi,C \in \mathcal{C}^*(\pi)} \inner{C_{\pi}}{Q-A} -\frac{1}{128}\|C\|_2^2 \geq \frac{6c_1}{n^2} r(t+k \log n+k^2) \right) \leq \exp(-t),$$for any $t>0$. Integrating over $t>0$ yields that \begin{equation}
\mathbb{E}\left[\sup_{ C \in \mathcal{C}^*(\pi), \pi} \left( \inner{C_{\pi}}{Q-A}-\frac{1}{128}\|C\|_2^2 \right)\right] \leq O \left( r( \frac{\log k}{n}+\frac{k^2}{n^2}  )\right) 
\end{equation}Therefore from (\ref{comb})
\begin{equation}\label{eq:25}
\mathbb{E}\left[\sup_{C, \pi} \left( \inner{C_{\pi}}{Q-A}-\frac{1}{128}\|C\|_2^2 \right)\right] \leq O(\frac{r}{n})+O \left( r( \frac{\log k}{n}+\frac{k^2}{n^2}  )\right)=O \left( r( \frac{\log k}{n}+\frac{k^2}{n^2}  )\right),
\end{equation}which completes the proof.

\end{proof}
\begin{proof}[Proof of Proposition \ref{prop}]

For any $B \in \mathcal{B}_r$, $\mathrm{Score}(B,A)=\max_{\pi} \left(\|A\|_2^2-\|A-B_{\pi}\|_2^2\right)$, where $\pi$ ranges over all equipartitions $\pi: [n] \rightarrow [k]$. Therefore, the equation (\ref{eq:condition}) implies that conditioning on $\mathcal{E}$ we have 
\begin{equation}
\min_{\pi} \|\hat{B}_{\pi}-A\|_2^2 \leq \min_{B \in \mathcal{B}_r, \pi }\| B_{\pi}-A\|_2^2+\nu^2,
\end{equation}By calling $\hat{\pi}$ the optimal permutation on the left hand side we conclude
\begin{equation}\label{eq:condition0}
\|\hat{B}_{\hat{\pi}}-A\|_2^2 \leq \min_{B \in \mathcal{B}_r, \pi }\| B_{\pi}-A\|_2^2+\nu^2,
\end{equation}

Now recall that $\|Q\|_{\infty} \leq r$ and $\mathcal{B}_r$ is the set of matrices with elements which are arbitrary multiples of $\frac{1}{n}$ in $[0,r]$. In particular for any $k$-equipartition of $[n]$ $\pi$, $Q(n,\pi) \in \mathcal{B}_r$. Since $\|Q(n,\pi)-Q(\pi)\|_{\infty} \leq \frac{1}{n}$, using the elementary $(u+v)^2 \leq 2(u^2+v^2)$ we have $$\|Q(n,\pi)_{\pi}-Q\|^2_{2} \leq 2 \|Q(\pi)_{\pi}-Q\|^2_{2} +\frac{2}{n^2}$$which implies by the definition of $Q(\pi)$,
\begin{equation}\label{bir} \|Q(n,\pi)_{\pi}-Q\|^2_{2} \leq 2\min_{B \in \mathbb{R}^{k \times k}, \pi }\| B_{\pi}-Q\|_2^2+\frac{2}{n^2}=2\hat{\epsilon}_k^{(O)}(Q)^2+\frac{2}{n^2}.\end{equation} 

Set $\tilde{\nu}^2=\nu^2+\frac{1}{n^2}$. Equation (\ref{eq:condition0}) since $Q(n,\pi) \in \mathcal{B}_r$ gives
\begin{align*}
\|\hat{B}_{\hat{\pi}}-A\|_2^2 \leq \|Q(n,\pi)_{\pi}-A\|_2^2+\nu^2,
\end{align*} almost surely conditional on $\mathcal{E}$, which now after adding and substracting $Q$ inside both the 2-norms and expanding them implies
\begin{align*}
\|\hat{B}_{\hat{\pi}}-Q\|_2^2 \leq \|Q(n,\pi)_{\pi}-Q\|_2^2+2\inner{Q(n,\pi)_{\pi}-\hat{B}_{\hat{\pi}}}{Q-A}+\nu^2.
\end{align*} almost surely conditional on $\mathcal{E}$. Adding and substracting $Q$ and $Q(\hat{\pi})_{\hat{\pi}}$ we have
\begin{align*}
\|\hat{B}_{\hat{\pi}}-Q\|_2^2 \leq \|Q(n,\pi)_{\pi}-Q\|_2^2+2\inner{Q(n,\pi)_{\pi}-Q}{Q-A}+2\inner{Q-Q(\hat{\pi})_{\hat{\pi}}}{Q-A}+2\inner{(Q(\hat{\pi})-\hat{B})_{\hat{\pi}}}{Q-A}+\nu^2,
\end{align*}almost surely conditional on $\mathcal{E}$, and therefore $\mathbb{E}\left[\|\hat{B}_{\hat{\pi}}-Q\|_2^2| \mathcal{E}\right]$ is at most \begin{align*}
&\|Q(n,\pi)_{\pi}-Q\|_2^2+2\mathbb{E}\left[\inner{Q(n,\pi)_{\pi}-Q}{Q-A}+\inner{Q-Q(\hat{\pi})_{\hat{\pi}}}{Q-A}| \mathcal{E}\right]+2\mathbb{E}\left[\inner{(Q(\hat{\pi})-\hat{B})_{\hat{\pi}}}{Q-A} | \mathcal{E}\right]+\nu^2.
\end{align*}
Bounding now the first two expected inner product terms according to Lemma \ref{lemma11} and the last according to Lemma \ref{lemma12} we obtain that the quantity $\mathbb{E}\left[\|\hat{B}_{\hat{\pi}}-Q\|_2^2 | \mathcal{E}\right]$ is at most\begin{align*}
&\frac{18}{16}\|Q(n,\pi)_{\pi}-Q\|_2^2+\frac{2}{16}\mathbb{E}\left[\|Q(\hat{\pi})_{\hat{\pi}}-Q\|_2^2+\|(Q(\hat{\pi})-\hat{B})_{\hat{\pi}}\|_2^2 |  \mathcal{E} \right]+ O \left( r( \frac{\log k}{n}+\frac{k^2}{n^2}  )\right)+\nu^2\\
&=\frac{9}{8}\|Q(n,\pi)_{\pi}-Q\|_2^2+\frac{1}{8}\mathbb{E}\left[\|Q(\hat{\pi})_{\hat{\pi}}-Q\|_2^2+\|(Q(\hat{\pi})-\hat{B})_{\hat{\pi}}\|_2^2 | \mathcal{E} \right] + O \left( r( \frac{\log k}{n}+\frac{k^2}{n^2}  )\right)+\nu^2\\
\end{align*}Set $\tilde{\nu}^2=\nu^2+\frac{1}{n^2}$. From (\ref{bir}) we conclude
\begin{align*}
\mathbb{E}\left[\|\hat{B}_{\hat{\pi}}-Q\|_2^2| \mathcal{E}\right] \leq \frac{1}{8}\mathbb{E}\left[\|Q(\hat{\pi})_{\hat{\pi}}-Q\|_2^2+\|(Q(\hat{\pi})-\hat{B})_{\hat{\pi}}\|_2^2 | \mathcal{E} \right] + O \left( \hat{\epsilon}_k^{(O)}(Q)^2+ r( \frac{\log k}{n}+\frac{k^2}{n^2}  )+\tilde{\nu}^2 \right) . 
\end{align*}

For any $k$-equipartition $\pi$, the matrix $Q(\pi)_{\pi}$ is the minimizer of the quantity $\|C_{\pi}-Q\|_2^2$ over all matrices $C \in \mathbb{R}^{k \times k}$. Therefore \begin{equation*}\|Q(\hat{\pi})_{\hat{\pi}}-Q\|_2 \leq \|\hat{B}_{\hat{\pi}}-Q\|_2,\end{equation*} almost surely, and by triangle inequality also \begin{equation*}\|(Q(\hat{\pi})-\hat{B})_{\hat{\pi}}\|_2 \leq\|\hat{B}_{\hat{\pi}}-Q\|_2+\|Q(\hat{\pi})_{\hat{\pi}}-Q\|_2 \leq  2\|\hat{B}_{\hat{\pi}}-Q\|_2,\end{equation*} almost surely. Hence combining the last inequalities together we obtain
\begin{align*}\mathbb{E}\left[\|\hat{B}_{\hat{\pi}}-Q\|_2^2 | \mathcal{E} \right]  \leq \frac{5}{8}\mathbb{E}\left[\|\hat{B}_{\hat{\pi}}-Q\|_2^2 | \mathcal{E} \right] + O \left( \hat{\epsilon}_k^{(O)}(Q)^2+ r( \frac{\log k}{n}+\frac{k^2}{n^2}) + \tilde{\nu}^2  \right)  
\end{align*}
or 
\begin{equation}\label{tilde}
\mathbb{E}\left[\|\hat{B}_{\hat{\pi}}-Q\|_2^2 | \mathcal{E} \right] =O\left(\hat{\epsilon}_k^{(O)}(Q)^2+r\left(\frac{\log k}{n}+\frac{k^2}{n^2}\right)+\tilde{\nu}^2  \right).
\end{equation}Finally notice that since $r \geq \frac{1}{n}$, $$\tilde{\nu}^2=\nu^2+\frac{1}{n^2} = O \left(\nu^2+ r\frac{\log k}{n}  \right).$$Plugging this in (\ref{tilde}) we obtain \begin{equation}\label{tilde}
\mathbb{E}\left[\|\hat{B}_{\hat{\pi}}-Q\|_2^2 | \mathcal{E} \right] =O\left(\hat{\epsilon}_k^{(O)}(Q)^2+r\left(\frac{\log k}{n}+\frac{k^2}{n^2}\right)+\nu^2  \right).
\end{equation}
This completes the proof of the Proposition \ref{prop}.
\end{proof}

We need the following Lemma from \cite{BorgsCS15}.

\begin{lemma}\label{lem8}(Lemma 8 in \cite{BorgsCS15})
Let $(\hat{\rho},\hat{B})$ be the output of Algorithm 1 from \cite{BorgsCS15} and $A$ the adjacency matrix of the observed graph $G$. Then under the assumptions of Theorem 1 the following properties hold with probability at least $1-2\exp(-\Omega(n\rho \epsilon))$;
\begin{itemize}
\item $|e(G)-\hat{\rho}| \leq \frac{\rho}{4}.$
\item If $d_{\max}(G) \leq \frac{\lambda \rho}{4}$ and $e(G) \geq \frac{\rho}{2}$, then $$\mathrm{Score}(\hat{B},A) \geq \max_{B \in \mathcal{B}_{\lambda \hat{\rho}}}\left[\mathrm{Score}(B,A)\right] -O\left(  \frac{\lambda^2 \hat{\rho}^2  k^2\log n}{n \epsilon}\right).$$
\end{itemize}
\end{lemma}

We also need the following lemma.

\begin{lemma}\label{riemman}
Let $B$ be $k \times k$ symmetric matrix with non negative entries and let $\pi$ be the standard equipartition of $[n]$ into $k$ classes. Then $$\|W[B_{\pi}]-W[B]\|_2 \leq \sqrt{\frac{10(k-1)}{n}} \|B\|_2.$$
\end{lemma}

\begin{proof}
If $k=1$ then $W[B_{\pi}]=W[B]$ and both the left and right hand side are zero. 

If $k \geq 2$ we observe that $\sqrt{\frac{10(k-1)}{n}} \|B\|_2>\sqrt{\frac{4k}{n}} \|B\|_2$ and the rest of the proof follows from Lemma 7 in \cite{BorgsCS15}.

\end{proof}

Now we present a proof of Theorem \ref{stronger}.
\begin{proof}[Proof of Theorem \ref{stronger}]

We first claim that with probability at least $1-ne^{-\Omega(\epsilon \rho n)}-O(\frac{\Lambda}{n})$ both conditions of the second part of Lemma \ref{lem8}, $d_{\max}(G) \leq \frac{\lambda \rho}{4}$ and $e(G) \geq \frac{\rho}{2}$, are satisfied. This follows since $\rho \log n \geq 6n$ and $ \Lambda \leq \lambda/8$ by the Lemmata 13 and 12 from \cite{BorgsCS15}, respectively.

Consider the event $\mathcal{E}=\{\mathrm{Score}\left(\hat{B},A\right) \geq \max_{B \in \mathcal{B}_r}[\mathrm{Score}\left(B,A\right)]-\nu^2\} \cap \{ \frac{1}{10} \rho \leq \hat{\rho} \leq 10 \rho \}$ for our $A$ which follows $\mathrm{Bern}_0(Q)$ for $Q=\rho H_n( W)$, and parameters $r=10 \lambda \rho$ and $\nu=O\left(\lambda \hat{\rho} \sqrt{\frac{k^2\log n}{n \epsilon}}\right)$. Using the second part of Lemma \ref{lem} we obtain that $$\mathbb{P}\left(\mathcal{E}\right)=1-ne^{-\Omega(\epsilon \rho n)}-O(\frac{\Lambda}{n})=1-O(\frac{\Lambda}{n}),$$ where for the last equality we used that $n \epsilon \rho/ \log n \rightarrow  + \infty$. Hence using also that on $\mathcal{E}$, $\|B\|_2 \leq \lambda \hat{\rho} \leq 10 \lambda \rho=r$. Therefore all the conditions of Proposition \ref{prop} are satisfied for the event $\mathcal{E}$.

Hence,
\begin{align*}
&\mathbb{E}\left[\hat{\delta}_2\left(\hat{B},\rho H_n(W)\right)^2 \right] = \mathbb{P}\left(\mathcal{E}\right) \mathbb{E}\left[\hat{\delta}_2\left(\hat{B},\rho H_n(W)\right)^2 | \mathcal{E} \right] + \mathbb{P}\left(\mathcal{E}^c\right) \mathbb{E}\left[\hat{\delta}_2\left(\hat{B},\rho H_n(W)\right)^2 | \mathcal{E}^c \right]\\
&\leq  \mathbb{E}\left[\hat{\delta}_2\left(\hat{B},\rho H_n(W)\right)^2 | \mathcal{E} \right]+O(\frac{\Lambda}{n})\mathbb{E}\left[\hat{\delta}_2\left(\hat{B},\rho H_n(W)\right)^2 | \mathcal{E}^c \right]\\
&\leq O\left(\mathbb{E}\left[\hat{\epsilon}_k^{(O)}(\rho H_n(W))^2\right]+\lambda \rho \left(\frac{\log k}{n}+\frac{k^2}{n^2}\right) + \lambda^2   \rho^2 \frac{k^2\log n}{n \epsilon}\right)+O \left( \hat{\rho}^2  \lambda^2 \frac{\Lambda}{n}\right),
\end{align*}where for the last inequality we used the crude bound $\hat{\delta}_2\left(\hat{B},\rho H_n(W)\right) \leq O(\lambda \hat{\rho})$ on its complement $\mathcal{E}^c$ and Proposition \ref{prop} on $\mathcal{E}$. Note that Proposition \ref{prop} can be applied because $\rho $ is assumed to be bigger than $\frac{6 \log n}{n}$. Now since $\hat{\rho}$ is $e(G)$ with Laplace noise of parameter $O(\frac{1}{n\epsilon})$ we can easily conclude that $\hat{\rho}$ is stochastically dominated by the addition of a $\frac{1}{n^2}\mathrm{Binom}(\binom{n}{2},\Lambda \rho)$ and an independent $\mathrm{Lap}(\frac{4}{n\epsilon})$. In particular $$\mathbb{E}\left[\hat{\rho}^2\right] \leq O(\Lambda^2\rho^2+\frac{\Lambda^2}{n^2 \epsilon^2})=O(\Lambda^2\rho^2),$$or $$ \mathbb{E} \left[\frac{\hat{\rho}^2}{\rho^2}\right]=O(\Lambda^2).$$ Hence rescaling by $\rho^2$ we have
\begin{align*}
&\mathbb{E}\left[\hat{\delta}_2\left(\frac{1}{\rho}\hat{B}, H_n(W)\right)^2 \right] \leq O\left(\mathbb{E}\left[\hat{\epsilon}_k^{(O)}( H_n(W))^2\right]+\lambda \left(\frac{\log k}{\rho n}+\frac{k^2}{\rho n^2}\right) + \lambda^2\frac{k^2\log n}{n \epsilon}+\lambda^2 \frac{\Lambda^3}{n}\right).
\end{align*}

Using triangle inequality for the $\ell_2$ norm and the elementary $(u+v)^2 \leq 2(u^2+v^2)$ we know that $$\mathbb{E}\left[\hat{\delta}_2\left(\frac{1}{\hat{\rho}}\hat{B}, H_n(W)\right)^2 \right]=O(\mathbb{E}\left[\hat{\delta}_2\left(\frac{1}{\rho}\hat{B}, H_n(W)\right)^2 \right]+\mathbb{E}\left[\|\frac{1}{\hat{\rho}}\hat{B}-\frac{1}{\rho} \hat{B}\|_2^2 \right]).$$Hence $\mathbb{E}\left[\hat{\delta}_2\left(\frac{1}{\hat{\rho}}\hat{B}, H_n(W)\right)^2 \right]$ is at most of the order 
\begin{align*}
 O\left(\mathbb{E}\left[\hat{\epsilon}_k^{(O)}( H_n(W))^2\right]+\lambda \left(\frac{\log k}{\rho n}+\frac{k^2}{\rho n^2}\right) + \lambda^2 \frac{k^2\log n}{n \epsilon}+\lambda^2 \frac{\Lambda^3}{n}+\mathbb{E}\left[\|\hat{B}\|^2_2 (\frac{1}{\hat{\rho}}-\frac{1}{\rho})^2\right]\right).
\end{align*}or since $\epsilon=O(k^2 \log n/\lambda^3)=O(k^2 \log n/\Lambda^3)$, at most of the order
 \begin{align}\label{step}
 O\left(\mathbb{E}\left[\hat{\epsilon}_k^{(O)}( H_n(W))^2\right]+\lambda \left(\frac{\log k}{\rho n}+\frac{k^2}{\rho n^2}\right) + \lambda^2 \frac{k^2\log n}{n \epsilon}+\mathbb{E}\left[\|\hat{B}\|^2_2 (\frac{1}{\hat{\rho}}-\frac{1}{\rho})^2\right]\right).
\end{align}

Now we focus on the term $\mathbb{E}\left[\|\hat{B}\|^2_2 (\frac{1}{\hat{\rho}}-\frac{1}{\rho})^2\right]$. since $\hat{B} \in \mathcal{B}_{\lambda \hat{\rho}}$ we have that $\|\hat{B}\|_2 \leq \lambda \hat{\rho}$ almost surely. Therefore, $$\mathbb{E}\left[\|\hat{B}\|^2_2 (\frac{1}{\hat{\rho}}-\frac{1}{\rho})^2\right] \leq \lambda^2  \mathbb{E}\left[(1-\frac{\hat{\rho}}{\rho})^2\right]=\frac{\lambda^2}{\rho^2}  \mathbb{E}\left[(\rho-\hat{\rho})^2\right].$$
Using that $\hat{\rho}=e(G)+\mathrm{Lap}(\frac{4}{n\epsilon})$ we conclude $$\mathbb{E}\left[\|\hat{B}\|^2_2 (\frac{1}{\hat{\rho}}-\frac{1}{\rho})^2\right] \leq \frac{\lambda^2}{\rho^2}  \left(\mathbb{E}\left[(\rho-e(G))^2\right]+O(\frac{1}{n^2 \epsilon^2}) \right).$$ Using Lemma 12 in \cite{BorgsCS15} we have that $\mathbb{E}\left[(\rho-e(G))^2\right]=O(\frac{\rho^2 \Lambda}{n})$ and therefore  $$\mathbb{E}\left[\|\hat{B}\|^2_2 (\frac{1}{\hat{\rho}}-\frac{1}{\rho})^2\right] \leq O  \left(\frac{\lambda^2 \Lambda}{n}+\frac{\lambda^2}{n^2 \epsilon^2 \rho^2} \right).$$Using that to (\ref{step}) we obtain that $\mathbb{E}\left[\hat{\delta}_2\left(\frac{1}{\rho}\hat{B}, H_n(W)\right)^2 \right]$ is at most of the order
 \begin{align*}
 O\left(\mathbb{E}\left[\hat{\epsilon}_k^{(O)}( H_n(W))^2\right]+\lambda \left(\frac{\log k}{\rho n}+\frac{k^2}{\rho n^2}\right) + \lambda^2 \frac{k^2\log n}{n \epsilon}+\frac{\lambda^2 \Lambda}{n}+\frac{\lambda^2}{n^2 \epsilon^2 \rho^2}\right).
\end{align*}

Using our assumptions  $\epsilon=O(k^2 \log n/\Lambda)$, $\frac{\lambda^2 \Lambda}{n}=O(\lambda^2 \frac{k^2\log n}{n \epsilon})$ and therefore we conclude that  \begin{align*}\mathbb{E}\left[\hat{\delta}_2\left(\frac{1}{\hat{\rho}}\hat{B}, H_n(W)\right)^2 \right] \leq
 O\left(\mathbb{E}\left[\hat{\epsilon}_k^{(O)}( H_n(W))^2\right]+\lambda \left(\frac{\log k}{\rho n}+\frac{k^2}{\rho n^2}\right) + \lambda^2 \frac{k^2\log n}{n \epsilon}+\frac{\lambda^2}{n^2 \rho^2 \epsilon^2}\right).
\end{align*}

Consider now $V$ a $k \times k$ matrix such that $\epsilon^{(O)}_k(W)=\|W-W[V]\|_2$. Since $B$ is obtained by averages of the values of $W$ over $k$ classes we have that $\|V\|_2 \leq \|W\|_2 \leq \sqrt{ \|W\|_{\infty} \|W\|_1} \leq \sqrt{\lambda}.$ Furthermore by Lemma \ref{riemman} for $\pi$ the standard $k$-equipartition of $[n]$ we have 
\begin{align} \label{step1} 
\epsilon^{(O)}_k(W) \geq \|W-W[V_{\pi}]\|_2-\sqrt{ \frac{10 \lambda (k-1)}{n}} \geq \hat{\delta}_2(V_{\pi},W)-\sqrt{ \frac{10 \lambda (k-1)}{n}} .
\end{align} Hence for the $V$ and $\pi$ chosen above,

\begin{align*}
\hat{\epsilon}^{(O)}_k(H_n(W)) &\leq \hat{\delta}_2(V,H_n(W))\\
&= \min_{\sigma \in \mathcal{S}_n}  \|V_{\pi}-H_n(W)_{\sigma}\|_2, \text{ using Lemma 6 in \cite{BorgsCS15}}\\
& \leq \|W[V_{\pi}]-W\|_2 +\hat{\delta}_2(H_n(W),W)\\
& \leq \epsilon^{(O)}_k(W)+\epsilon_n(W)+O(\sqrt{ \frac{ \lambda (k-1)}{n}}), \text{ using \ref{step1}}
\end{align*}

 Now using this bound and the elementary $(u+v+w)^2\leq 3(u^2+v^2+w^2)$ we obtain that $\mathbb{E}\left[\delta_2\left(\frac{1}{\hat{\rho}}\hat{B}, H_n(W)\right)^2 \right]$ is at most  \begin{align*}O\left(\mathbb{E}\left[\epsilon_k^{(O)}( W)^2+\epsilon_n(W)^2\right]+\lambda\frac{k-1}{n}+\lambda \left(\frac{\log k}{\rho n}+\frac{k^2}{\rho n^2}\right) + \lambda^2 \frac{k^2\log n}{n \epsilon}+\frac{\lambda^2}{n^2 \rho^2 \epsilon^2}\right). \end{align*} Using now triangle inequality for $\delta_2$ we have $$\delta_2\left(\frac{1}{\hat{\rho}}\hat{B}, W\right) \leq \delta_2\left(\frac{1}{\hat{\rho}}\hat{B},H_n(W)\right)+\delta_2\left(H_n(W),W\right)=\delta_2\left(\frac{1}{\rho}\hat{B}, H_n(W)\right)+\epsilon_n(W)$$almost surely and therefore $\mathbb{E}\left[\delta_2\left(\frac{1}{\hat{\rho}}\hat{B}, W\right)^2 \right] $ is at most  of the order \begin{align*} O\left(\mathbb{E}\left[\epsilon_k^{(O)}( W)^2+\epsilon_n(W)^2\right]+\lambda \frac{k-1}{n}+\lambda \left(\frac{\log k}{\rho n}+\frac{k^2}{\rho n^2}\right) + \lambda^2 \frac{k^2\log n}{n \epsilon}+\frac{\lambda^2}{n^2 \rho^2 \epsilon^2}\right). \end{align*}This completes the proof of Theorem \ref{stronger}.

\end{proof}

\begin{proof}[Proof of Corollary \ref{coro1}]
Since $k \geq 2$ and $\rho n \geq k-2$ the lower bound follows from Theorem 3 from~\cite{McMillanS17} for $r=\Lambda \rho$. For the upper bound we use Algorithm 1 from \cite{BorgsCS15} and the output $k$-block graphon defined by $\frac{1}{\hat{\rho}}\hat{B}$. For any $k$-block graphon and $\lambda=O(\Lambda)$, Theorem \ref{stronger} implies that  $\mathbb{E}\left[\delta_2\left(\frac{1}{\hat{\rho}}\hat{B}, W \right)^2 \right]$ is at most  of the order
 \begin{align}\label{bound1}
O\left(\mathbb{E}\left[\epsilon_k^{(O)}( W)^2+\epsilon_n(W)^2\right]+\Lambda \frac{k-1}{n}+\Lambda \left(\frac{\log k}{\rho n}+\frac{k^2}{\rho n^2}\right) + \Lambda^2 \frac{k^2\log n}{n \epsilon}+\frac{\Lambda^2}{n^2 \rho^2 \epsilon^2}\right). 
\end{align} Now, since $W$ is a $k$-block graphon we have $\epsilon_k^{(O)}( W)=0$. Furthermore by Lemma 14 in \cite{BorgsCS15}, $\mathbb{E}\left[\epsilon_n(W)^2\right] \leq O\left(\Lambda^2 \sqrt{\frac{k}{n}}\right)$. Plugging both these equalities in (\ref{bound1}) and using $\Lambda \frac{k-1}{n}=O\left(\Lambda^2 \sqrt{\frac{k}{n}}\right)$ we conclude that for any $k$-block graphon $\mathbb{E}\left[\delta_2\left(\frac{1}{\hat{\rho}}\hat{B}, W \right)^2 \right] $ is at most  of the order 
\begin{align*} O\left(\Lambda^2 \sqrt{\frac{k}{n}}+\Lambda \left(\frac{\log k}{\rho n}+\frac{k^2}{\rho n^2}\right) + \Lambda^2 \frac{k^2\log n}{n \epsilon}+\frac{\Lambda^2}{n^2 \rho^2 \epsilon^2}\right). 
\end{align*} This completes the proof of the upper bound and the Corollary.

\end{proof}

\section{Proof of Theorem \ref{kbiglower}}\label{kbig}

\begin{proof}For simplicity we present the proof only in the case $k=2$. The argument generalizes easily to higher $k$.

We claim that it suffices to be established that one cannot estimate the density of any $k$-block graphon $W$, $\|W\|_1$, with mean squared error of the order $O(\max (\frac{1}{n},\frac{1}{\epsilon^2 n^2}))$. We establish this claim by contradiction. Suppose that we have established the above and that the rate $\tilde{R}_k(\epsilon,n)$ is $O\left(\frac{1}{n^2\epsilon^2}\right)$. That implies the existence of an $\epsilon$-node-DP algorithm, which for any $W$ $k$-block graphon, $$\mathbb{E}_{ G \sim G_n(\rho W), \hat{B} \sim \mathcal{A}_G} [ \delta_2(\hat{B},W)^2]=O\left(\frac{1}{n^2\epsilon^2}\right).$$Note though that by Cauchy-Schawrz and triangle inequality,
$$\delta_2(\hat{B},W)=\min_{\phi} \|W[\hat{B}]_{\phi}-W\|_2 \geq \min_{\phi} \|W[\hat{B}]_{\phi}-W\|_1 \geq | \|W[\hat{B}]_{\phi}\|_1-\|W\|_1|,$$where $\phi$ ranges over all measure-preserving transformations of $[0,1]$. In particular, for any $W$ $k$-block graphon, $$\mathbb{E}_{ G \sim G_n(\rho W), \hat{B} \sim \mathcal{A}_G} [| \|W[\hat{B}]_{\phi}\|_1-\|W\|_1|^2]=O(\left(\frac{1}{n^2\epsilon^2}\right).$$This is a contradiction with the assumption that we cannot approximate the density of $W$ at a $O(\max (\frac{1}{n},\frac{1}{\epsilon^2 n^2}))$ level. The proof of the claim is complete.  

Now we proceed with the lower bound on the density estimation by giving a reduction to (regular) differentially private estimation of the secret parameter $q$ given
  $n$ samples from a Bernouilli distribution (that is, $n$ biased
  coins, each of which is heads independently with probability $q$). 

  Given a parameter $q\in [0,1]$, let $W_q:[0,1]^2\to [0,1]$ be the
  graphon given by 
  $$W_q(x,y) =
  \begin{cases}
    1 & \text{if } x, y \leq q \text{ or } x,y \geq q\, , \\
    0 & \text{otherwise.}
  \end{cases}$$
  This is a 2-block graphon with blocks of sizes $q$ and $1-q$,
  respectively. The graphs generated from $W_q$ consist of two
  cliques (of size roughly $qn$ and $(1-q)n$) so it is easy to know
  which vertices belong to the same block. 

  The density of $W_q$ is $\tau(q)= 
  q^2+ (1-q)^2 = \frac 1 2 + 2 (q-\frac 1 2)^2$. Consider an
  algorithm $A$ that, given $G\sim G_n(W)$, aims to estimate the density $\tau(q)$
  of $W_q$. We can use its output (call it $\hat \tau$) to estimate $q$ by
  setting $\hat q = \frac 1 2 - \sqrt{\frac{\hat \tau - \frac 1
      2}{2}}$. This function's derivative is finite and nonzero as
  long as $\hat \tau$ is bounded away from $\frac 1 2$. Thus, an
  algorithm that can estimate $\tau(q)$ within error $\alpha$ on samples
  from $W_q$ can be used to estimate $q$ up to error $O(\alpha)$ (as
  long as $q$ is bounded away from 1/2). 

  To reduce to estimation of the Bernouilli parameter, suppose we are
  given a sample $X = (X_1,X_2,...,X_n)$ of size  $n$, drawn
  i.i.d. from Bernouilli($q$) for uknown $q$. We may generate a graph $G(X)$
  by creating two cliques of size $N_0$ and $N_1$, respectively, where
  $N_1 = \sum_i X_i$ (the number of ones in $X$) and $N_0 =
  n-N_1$. The distribution of $G(X)$ is exactly $G_n(W_q)$, and so we
  can run our density estimation algorithm $A$ to get an estimate $\hat
  \tau$ of the density of $W_q$ and use that to compute $\hat q$, an
  estimate of $q$. 

  Observe that, if $A$ is $\eps$-node-differentially private, then the
  composed algorithm
  $A(G(\cdot))$ is $\eps$-differentially private with respect to its
  input (from
  $ \{0,1\}^n$). To see why, note that changing one bit of the string $x$
  changes the edges of exactly one vertex in $G(x)$ (corresponding to
  a change in the clique to which it is assigned). Since $A$ is
  $\eps$-node-differentially private, a change in one bit of $x$
  yields a change of at most $\eps$ in the distribution of
  $A(G(x))$. 

 Fix a constant $c>0$ and consider the distribution $P$ obtained by choosing $q$ uniformly in
 ${\frac 1 4, \frac 14 +\alpha}$, where $\alpha = c\max (\frac{1}{\sqrt{n}},
  \frac 1 {\epsilon n})$ is the desired error bound, and outputing
  $W_q$. One can pick $c$ so that there is no $\eps$-differentially private
  algorithm that can distinguish the corresponging Bernouilli
  distributions (with $q \in {\frac 1 4, \frac 14 +\alpha}$) with
  probability better than $0.9$.

  An $\eps$-node-DP algorithm for estimating the density with error
  $o(\alpha)$ could be used to create an algorithm for Bernouilli
  estimation with error $o(\alpha)$, which would in turn allow one to
  estimate $q$ with error $o(\alpha)$, yielding a contradiction.
\end{proof}

\section{Proof of Theorem \ref{extension}}\label{ext}
  We start with a lemma.
\begin{lemma}\label{lem}
Let $\mu$ be a probability measure on $\Omega$ and $\mathcal{A}'$ be a randomized algorithm designed for input from $\mathcal{H}' \subseteq \mathcal{M}$. Suppose that for any $D \in \mathcal{H}'$, $\mathcal{A}'(D)$ is absolutely continuous to $\mu$ and let $f_D$ the Radon-Nikodym derivative $\frac{d\mathcal{A}'(D)}{d\mu}$. Then the following are equivalent \begin{itemize} \item[(1)] $\mathcal{A}'$ is $\epsilon$-differentially private; \item[(2)] For any $D,D' \in \mathcal{H}$ \begin{equation}\label{prime}  f_{\mathcal{A}'(D)} \leq \exp \left(\epsilon d(D,D') \right) f_{\mathcal{A}'(D')}, \end{equation}$\mu$-almost surely. 
\end{itemize}
\end{lemma}

\begin{proof}
For the one direction, suppose $\mathcal{A}'$ satisfies (\ref{prime}). Then for any set $S \in \mathcal{F}$ we obtain
 \begin{align*} \mathbb{P}\left(\mathcal{A}'(D) \in S \right) &=\int_{S}  f_{\mathcal{A}'(D)} d\mu \\
& \leq \exp \left(\epsilon d(D,D') \right) \int_{ S} f_{\mathcal{A}'(D')} d\mu\\
&= \exp \left(\epsilon d(D,D') \right)\mathbb{P}\left(\mathcal{A}'(D) \in S \right).
 \end{align*}We prove the other direction by contradiction. Consider the set $$S= \{ f_{\mathcal{A}'(D)} > \exp \left(\epsilon d(D,D') \right) f_{\mathcal{A}'(D')} \} \in \mathcal{F}$$ and assume that $\mu (S)>0$. By definition on being strictly positive on a set of positive measure \begin{equation*}\int_S \left[f_{\mathcal{A'}(D)} -\exp \left(\epsilon d(D,D') \right) f_{\mathcal{A}'(D')}\right]  d \mu>0 \end{equation*}or equivalently\begin{equation}\label{contrad}\int_S f_{\mathcal{A'}(D)}d \mu > \exp \left(\epsilon d(D,D') \right) \int_Sf_{\mathcal{A}'(D')}  d \mu. \end{equation} On the other hand using $\epsilon$-differential privacy we obtain
\begin{align*}  \int_S f_{\mathcal{A'}(D)}d \mu & =\mathbb{P}(\mathcal{A'}(D) \in S) \\\
&\leq \exp \left(\epsilon d(D,D') \right)\mathbb{P}(\mathcal{A'}(D') \in S) \\
&= \exp \left(\epsilon d(D,D') \right) \int_Sf_{\mathcal{A}'(D')}  d \mu,
\end{align*} a contradiction with (\ref{contrad}). This completes the proof of the Lemma.
\end{proof}
Now we establish Theorem \ref{extension}.

\begin{proof}

Since $\mathcal{H} \not  = \emptyset$, let $D_0 \in \mathcal{H}$ and denote by $\mu$ the measure $\hat{\mathcal{A}}(D_0)$. From the definition of differential privacy we know for all $D \in \mathcal{H}$ and $S \in \mathcal{F}$, if $\mathbb{P}\left( \hat{\mathcal{A}}(D_0) \in S\right) =0$ then $\mathbb{P}\left( \hat{\mathcal{A}}(D) \in S\right) =0$. In the language of measure theory that means the measure $\hat{\mathcal{A}}(D)$ is absolutely continuous to $\mathcal{A}(D_0)$. By Radon-Nikodym theorem we conclude that there are measurable functions $f_D: \Omega \rightarrow [0,+\infty)$ such that for all $S \in \mathcal{F}$, \begin{equation}\mathbb{P}\left( \hat{\mathcal{A}}(D) \in S\right) = \int_S f_D d\mu. \end{equation}

 We define now the following randomized algorithm $\mathcal{A}$. For every $D \in \mathcal{M}$, $\mathcal{A}(D)$ samples from $\Omega$ according to the absolutely continuous to $\mu$ distribution with density proportional to $$\inf_{D' \in \mathcal{H}} \left[ \exp\left(\epsilon d(D,D')\right) f_{\hat{\mathcal{A}}(D')} \right].$$ That is for every $\omega \in \Omega$ its density with respect to $\mu$  is defined as \begin{equation*}f_{\mathcal{A}(D)}(\omega) =\frac{1}{Z_D} \inf_{D' \in \mathcal{H}} \left[ \exp\left(\epsilon d(D,D')\right) f_{\hat{\mathcal{A}}(D')}(\omega) \right],\end{equation*} where \begin{equation*}Z_D:=\int_{\Omega} \inf_{D' \in \mathcal{H}} \left[ \left(\epsilon d(D,D')\right) f_{\hat{\mathcal{A}}(D)'} \right] d \mu.\end{equation*}In particular for all $S \in \mathcal{F}$ it holds $$\mathbb{P}(\mathcal{A}(D) \in S)=\int_S f_{\mathcal{A}(D)}d \mu.$$

We first prove that $\mathcal{A}$ is $2\epsilon$-differentially private over all pairs of input from $\mathcal{M}$. Using Lemma \ref{lem} it suffices to prove that for any $D_1,D_2 \in \mathcal{H}$,
 \begin{align*}
f_{\mathcal{A}(D_1)} \leq \exp \left(2 \epsilon d(D_1,D_2)\right) f_{\mathcal{A}(D_2)},
\end{align*} $\mu$-almost surely. We establish it in particular for every $\omega \in \Omega$. Let $D_1,D_2 \in \mathcal{M}$. Using triangle inequality we obtain for every $\omega \in \Omega$, \begin{align*}\inf_{D' \in \mathcal{H}} \left[ \exp \left( \epsilon d(D_1,D')\right) f_{\hat{\mathcal{A}}(D')}(\omega) \right] & \leq \inf_{D' \in \mathcal{H}} \left[ \exp\left(\epsilon \left[d(D_1,D_2)+d(D_2,D')\right] \right) f_{\hat{\mathcal{A}}(D')}(\omega) \right]\\
&=\exp \left( \epsilon d(D_1,D_2)\right) \inf_{D' \in \mathcal{H}} \left[ \exp \left(\epsilon d(D,D')\right) f_{\hat{\mathcal{A}}(D')}(\omega) \right],
\end{align*}which implies that for any $D_1,D_2 \in \mathcal{M}$, 
\begin{align*}
Z_{D_1}&=\int_{\Omega} \inf_{D' \in \mathcal{H}} \left[ \exp\left(\epsilon d(D_1,D')\right) f_{\hat{\mathcal{A}}(D')} \right] d \mu \\
&\leq \exp\left(\epsilon d(D_1,D_2) \right)\int_{\Omega} \inf_{D' \in \mathcal{H}} \left[ \exp \left( \epsilon d(D_2,D') \right) f_{\hat{\mathcal{A}}(D')}(\omega) \right] d\mu\\
&=\exp \left( \epsilon d(D_1,D_2) \right) Z_{D_2}. 
\end{align*}Therefore using the above two inequalities we obtain that for any $D_1,D_2 \in \mathcal{H}$ and $\omega \in \Omega$,
 \begin{align*}
f_{\mathcal{A}(D_1)}(\omega) &=\frac{1}{Z_{D_1}} \inf_{D' \in \mathcal{H}} \left[ \exp \left( \epsilon d(D_1,D') \right) f_{\hat{\mathcal{A}}(D')}(\omega) \right] \\
&\leq \frac{1}{\exp\left(-\epsilon d(D_2,D_1)\right)Z_{D_2}}\exp\left(\epsilon d(D_1,D_2)\right) \inf_{D' \in \mathcal{H}}  \left[ \exp \left( \epsilon d(D_2,D') \right) f_{\hat{\mathcal{A}}(D')}(\omega) \right] \\
&=\exp \left(2 \epsilon d(D_1,D_2)\right) \frac{1}{Z_{D_2}} \inf_{D' \in \mathcal{H}}  \left[ \exp \left( \epsilon d(D_2,D') \right) f_{\hat{\mathcal{A}}(D')}(\omega) \right]\\
&=\exp \left(2 \epsilon d(D_1,D_2)\right) f_{\mathcal{A}(D_2)}(\omega),
\end{align*}as we wanted. 

Now we prove that for every $D \in \mathcal{H}$, $\mathcal{A}(D) \overset{d}{=}  \hat{\mathcal{A}}(D)$. Consider an arbitrary $D \in \mathcal{H}$. From Lemma \ref{lem} we obtain that $\hat{\mathcal{A}}$ is $\epsilon$-differentially private which implies that for any $D,D' \in \mathcal{H}$ \begin{equation}  f_{\hat{\mathcal{A}}(D)} \leq \exp \left(\epsilon d(D,D') \right) f_{\hat{\mathcal{A}}(D')},   \end{equation} $\mu$-almost surely. Observing that the above inequality holds as $\mu$-almost sure equality if $D'=D$ we obtain that for any $D \in \mathcal{H}$ it holds $$f_{\hat{\mathcal{A}}(D)}(x)=\inf_{D' \in \mathcal{H}}  \left[ \exp \left( \epsilon d(D,D') \right) f_{\hat{\mathcal{A}}(D')}(x) \right],$$ $\mu$-almost surely. Using that $f_{\hat{\mathcal{A}}(D)}$ is the Radon-Nikodym derivative $\frac{d \hat{\mathcal{A}}(D)}{d \mu}$  we conclude $$Z_{D}:=\int_{\Omega} f_{\hat{\mathcal{A}}(D)} d\mu=\mu(\Omega)=1.$$ Therefore $$f_{\hat{\mathcal{A}}(D)}=\frac{1}{Z_D}\inf_{D' \in \mathcal{H}}  \left[ \exp \left( \epsilon d(D,D') \right) f_{\hat{\mathcal{A}}(D')} \right],$$ $\mu$-almost surely and hence $$ f_{\hat{\mathcal{A}}(D)}=f_{\mathcal{A}(D)},$$ $\mu$-almost surely. This suffices to conclude that $\hat{\mathcal{A}}(D) \overset{d}{=}  \mathcal{A}(D)$ as needed. 

The proof of Theorem \ref{extension} is complete.
\end{proof}

\section{The Proof of the $n^{\frac{3}{2}}$-Upper Bound}\label{ER}

\paragraph{Definitions and Notation}
 
Recall that $\mathcal{G}_n$ is the set of all undirected graphs on the vertex set $[n]$. For any $\rho \in [0,1]$ let $\mathcal{G}_{n \rho}$ the set of all undirected graphs on $n$ vertices with edge density at most $\rho$, $$\mathcal{G}_{n,\rho}=\{G \in \mathcal{G}_n | e(G) \leq \rho \}.$$For any $\emptyset \not = S \subseteq [n]$, $t,\rho \in [0,1]$ and $C>0$, $$A_{\rho,S,C}(t):=\{G \in \mathcal{G}_{n,\rho} \bigg{|} |E(S,S^c)+E(S)-t\left[ k\left(n-k\right)+\binom{k}{2} \right]| \leq C\max\{\sqrt{\rho},\sqrt{\frac{\log n}{n}}\}k\sqrt{n \log n} |\}.$$Consider the ``homogeneity" set $$\mathcal{H}_{\rho,C}=\bigcap_{\emptyset \not = S \subseteq [n]} A_{\rho,C,S}(e(G)).$$

\paragraph{Auxilary Lemmata}

\begin{lemma}\label{Stir}
Let $p \in [0,1] \cap \mathbb{Q}$. Suppose $N$ grows to infinity constrained to $Np \in \mathbb{Z}$. Then $$\binom{N}{Np} \exp(-N H_2(p))=\Omega \left( \frac{1}{\sqrt{N}}\right),$$where $H_2(p)=-p\log p-(1-p)\log (1-p).$
\end{lemma}

\begin{proof}
We have $$\binom{N}{Np}=\frac{N!}{(Np)!(N(1-p))!}.$$ By Stirling approximation $n!=\Theta( (\frac{n}{e})^{n} \sqrt{2 \pi n}).$ Therefore

\begin{align*}
\binom{N}{Np}&=\Theta\left( \frac{(\frac{N}{e})^{N} \sqrt{2 \pi N}}{ \left[ (\frac{Np}{e})^{Np} \sqrt{2 \pi Np}(\frac{N(1-p)}{e})^{N(1-p)} \sqrt{2 \pi N(1-p)} \right] }\right)\\
&=\Theta \left(  \frac{1}{2\sqrt{Np(1-p)}}\exp\left(-N\left(p\log p+(1-p) \log (1-p)\right)\right) \right)\\
&=\Omega \left(\frac{1}{\sqrt{N}}\right) \exp(NH_2(p)).
\end{align*}or $$\binom{N}{Np} \exp(-N H_2(p))=\Omega \left( \frac{1}{\sqrt{N}}\right).$$ The proof of the Lemma is complete.

\end{proof}

\begin{lemma}\label{triang}
For any $a,b>0$ the function $f: \mathbb{R} \rightarrow \mathbb{R}$, with $f(x)=\min\{a|x|,b\}$, for all $x \in \mathbb{R}$, satisfies the triangle inequality, $f(x+y) \leq f(x)+f(y)$ for all $x,y \in \mathbb{R}$.
\end{lemma}

\begin{proof}
Let $x,y \in \mathbb{R}$.  We distinguish two cases.

 If $a|x| \geq b$, or $a|y| \geq b$ then $$f(x+y) \leq b \leq \min\{a|x|,b\} +\min\{a|y|,b\} = f(x)+f(y).$$If $b \leq a|x|$ and $|a|y \leq b$ then $$f(x+y) \leq a|x+y| \leq a|x|+a|y| = \min\{a|x|,b\}+\min\{a|y|,b\}=f(x)+f(y).$$This completes the proof of the Lemma.
\end{proof}

\paragraph{Main Lemmata}
\begin{lemma}\label{union}
Let $\rho \in [0,1]$, $m \in \mathbb{N}$ with $m \leq \rho \binom{n}{2}$, $p=\frac{m}{\binom{n}{2}} \in [0,\rho]$ and $C>48$. Then it holds
$$\mathbb{P}_{G \sim G_{n,p}}\left[ \mathcal{H}^c_{\rho,C} \right] =O(\frac{1}{n^{\frac{C}{16}-2}}).$$ Furthermore $$\mathbb{P}_{G \sim G(n,m)}\left[ \mathcal{H}^c_{\rho,C} \right] = O\left(\frac{1}{n^{\frac{C}{16}-3}}\right).$$
\end{lemma}

\begin{proof} We start with proving that
\begin{equation}\label{first}
\mathbb{P}_{G \sim G_{n,p}}\left[ \mathcal{H}^c_{\rho,C} \right] \leq \mathbb{P}_{G \sim G_{n,p}}\left(|e(G)-p| > \frac{C \sqrt{\rho  \log n}}{2\sqrt{n}}\right)+\mathbb{P}_{G \sim G_{n,p}}\left(\bigcup_{\emptyset \not = S \subseteq [n]}A^c_{\rho,S,\frac{C}{2}}(p)\right).\end{equation}To establish (\ref{first}) it suffices to establish that $$\mathbb{P}_{G \sim G_{n,p}}\left(\mathcal{H}^c_{\rho,C} \cap \{|e(G)-p| \leq \frac{C \sqrt{\rho  \log n}}{2\sqrt{n}} \}\right) \leq \mathbb{P}_{G \sim G_{n,p}}\left(\bigcup_{\emptyset \not = S \subseteq [n]}A^c_{\rho,S,\frac{C}{2}}(p)\right).$$We establish the corresponding set inclusion. Let $G  \in \mathcal{H}^c_{\rho,C} \cap \{|e(G)-p| \leq C\frac{ \sqrt{\rho  \log n}}{2\sqrt{n}} \}.$ Since $G  \in \mathcal{H}^c_{\rho,C}$ there exists $\emptyset \not = S \subseteq [n]$ with $G \in A^c_{\rho,S,C}(e(G))$ and therefore it holds 
\begin{equation}\label{cut1} E(S,S^c)+E(S)-e(G)\left[ k\left(n-k\right)+\binom{k}{2} \right]| > C\max\{\sqrt{\rho},\sqrt{\frac{\log n}{n}}\}k\sqrt{n \log n} |.
\end{equation}Since $|e(G)-p| \leq  \frac{C \sqrt{\rho  \log n}}{2\sqrt{n}}$ we obtain \begin{align*} (p-e(G))\left[ k\left(n-k\right)+\binom{k}{2} \right] & \leq \frac{C \sqrt{\rho  \log n}}{2\sqrt{n}}\left[ k\left(n-k\right)+\binom{k}{2} \right] \\
& \leq \frac{C \sqrt{\rho  \log n}}{2\sqrt{n}}kn\\
&= \frac{C}{2}k \sqrt{\rho} \sqrt{n \log n}\\
& \leq \frac{C}{2}  \max\{\sqrt{\rho},\sqrt{\frac{\log n}{n}}\}k\sqrt{n \log n} . \end{align*} Rearranging we have $$\frac{C}{2}  \max\{\sqrt{\rho},\sqrt{\frac{\log n}{n}}\}k\sqrt{n \log n} -p \left[ k\left(n-k\right)+\binom{k}{2} \right] \geq -e(G)\left[ k\left(n-k\right)+\binom{k}{2} \right].$$ Using this into (\ref{cut1}) we obtain $$E(S,S^c)+E(S)-p\left[ k\left(n-k\right)+\binom{k}{2} \right]> \frac{C}{2}\max\{\sqrt{\rho},\sqrt{\frac{\log n}{n}}\}k\sqrt{n \log n} ,$$ which means that indeed $$G \in A^c_{\rho,S,\frac{C}{2}}(p)  \subseteq \bigcup_{\emptyset \not = S \subseteq [n]}A^c_{\rho,S,\frac{C}{2}}(p).$$The proof of (\ref{first}) is complete. 

Now by Hoeffding inequality since $p \in [0, \rho]$, we have  \begin{equation}\label{Hoeffding} \mathbb{P}_{G \sim G_{n,p}}\left(|e(G)-p| > \frac{C \sqrt{\rho  \log n}}{2\sqrt{n}}\right) \leq 2\exp(-\Omega(C^2 n \log n)).\end{equation} Combining (\ref{Hoeffding}) with (\ref{first}) we conclude that to establish our result for the $G_{n,p}$ it suffices to establish 
\begin{equation}\label{second}
 \mathbb{P}_{G \sim G_{n,p}}\left(\bigcup_{\emptyset \not = S \subseteq [n]}A^c_{\rho,S,\frac{C}{2}}(p)\right) =O(\frac{1}{n^{\frac{C}{16}-2}}).
\end{equation}

Now we recall that Bernstein inequality implies that for $N \in \mathbb{N}$, $Z $ distributed according to a $\mathrm{Bin}(N,p)$ and $t>0$, \begin{equation}\label{Bern}\mathbb{P}\left[ |Z-Np| \geq t \right] \leq 2\exp(-\frac{\frac{t^2}{2}}{Np+\frac{1}{3}t}).\end{equation}

Set $c_0=\max\{\sqrt{\rho},\sqrt{\frac{\log n}{n}}\}$ and fix some $\emptyset \not = S \subseteq V(G)$. The random variable $E(S,S^c)+E(S)$ is distributed according to a $\mathrm{Bin}(N,p)$ for $N=k\left(n-k\right)+\binom{k}{2}$. Therefore from (\ref{Bern}) we have \begin{equation}\label{Bern2} \mathbb{P}\left[A^c_{\rho,S, \frac{C}{2}}(p)\right] \leq 2\exp\left( -\frac{C^2}{4}\frac{c_0^2k^2n \log n/2}{\left(k(n-k)+\binom{k}{2}\right)p+Cc_0k \sqrt{n \log n}/6} \right).  \end{equation}
Now observe that since $C>1$, $$\left(k(n-k)+\binom{k}{2}\right)p \leq kn p \leq Cknp \leq Ckn \rho, $$ and by definition of $c_0$,  $$ c_0k \sqrt{n \log n}/3 \leq \max\{k\sqrt{ \rho n \log n},k \log n\}.$$ Furthermore notice that for any $\rho \in [0,1]$, $$k\sqrt{\rho n \log n} \leq  \max\{\rho kn ,k \log n\}.$$ Therefore, using again that $C>1$, $$\left(k(n-k)+\binom{k}{2}\right)p+ c_0k \sqrt{n \log n}/3 \leq 2C\max\{kn\rho,k \log n\}.$$Hence
\begin{align*} 
\frac{c_0^2k n /2}{\left(k(n-k)+\binom{k}{2}\right)p+Cc_0k \sqrt{n \log n}/6} & \geq \frac{ c_0^2k n /2 }{ 2C\max\{kn\rho,k \sqrt{\rho n \log n},k \log n\}} \\
& \geq  \frac{ \max\{\rho kn ,k \log n\}/2}{ 2C\max\{kn\rho,k \sqrt{\rho n \log n},k \log n\}}, \text{ from the definition of } c_0, \\
& \geq \frac{1}{4C}.
\end{align*} Using the last inequality in (\ref{Bern2}) we conclude that for any $\emptyset \not = S \subseteq V(G)$ \begin{equation}\label{Bern3} \mathbb{P}\left[A^c_{S, \rho, \frac{C}{2}}\right] \leq 2\exp\left( -Ck \log n/16 \right)=2n^{-\frac{Ck}{16}}.  \end{equation}
Using a union bound we obtain
\begin{align*}
&\mathbb{P}_{G \sim G_{n,p}}\left[ \bigcup_{S \subseteq V(G) } A^c_{S, \rho, \frac{C}{2}} \right] \\
&\leq 2\sum_{k=1}^n \binom{n}{k} n^{- \frac{Ck}{16}} \\
&\leq 2\sum_{k=1}^n n^k n^{- \frac{Ck}{16}} \\
&\leq 2n\frac{1}{n^{\frac{C}{16}-1}}, \text{ (using that } C>16)\\
&=\frac{2}{n^{\frac{C}{16}-2}}\\
&=O \left(\frac{1}{n^{\frac{C}{16}-2}}\right).
\end{align*}

This completes the proof for the Erdos Renyi case.

For the $G(n,m)$ case, we first recall that for any $p_1 \in [0,1]$ and $m_1 \in \mathbb{Z}$ the distribution of an Erdos Renyi graph with parameter $p_1 $, conditional on having exactly $m_1$ edges, is a sample from $G(n,m_1)$. Therefore using the tower property and the property already established for the Erdos Renyi case we conclude 
\begin{equation} 
\bE_{G \sim G_{n,p}}\left[ \mathbb{P}_{G_1 \sim G_{n,E(G)}}\left[ \mathcal{H}^c_{\rho,C}\right] \right]=\mathbb{P}_{G \sim G_{n,p}}\left[\mathcal{H}^c_{\rho,C}\right] \leq 2\frac{1}{n^{\frac{C}{16}-2}}.
\end{equation} Using Markov's inequality we obtain  

\begin{equation}\label{finbound} 
 \mathbb{P}_{G_1 \sim G(n,m)}\left[ \mathcal{H}^c_{\rho,C} \right] \mathbb{P}_{G \sim G_{n,p}}\left[E(G)=m\right]\leq O \left( \frac{1}{n^{\frac{C}{16}-2}}\right).
\end{equation}

 Since $p=\frac{m}{\binom{n}{2}}$ and $E(G)$ is distributed according to a binomial $\mathrm{Bin}(\binom{n}{2},p)$ we have for $N=\binom{n}{2}$,
 \begin{align*}\mathbb{P}_{G \sim G_{n,p}}\left[E(G)=m\right]&=\mathbb{P}_{Z \sim \mathrm{Bin}(N,p)}\left[Z=Np\right]\\
&= \binom{N}{pN} p^{Np}(1-p)^{N(1-p)} \\
&= \binom{N}{pN} \exp(-NH_2(p)) \\
& =\Omega(\frac{1}{\sqrt{N}}) \text{ (using Lemma \ref{Stir})} \\
& =\Omega(\frac{1}{n}).
\end{align*} Therefore \begin{equation}\label{finbound1} 
 \mathbb{P}_{G_1 \sim G(n,m)}\left[ \mathcal{H}^c_{\rho,C} \right] \Omega(\frac{1}{n}) \leq  \mathbb{P}_{G_1 \sim G(n,m)}\left[ \mathcal{H}^c_{\rho,C} \right]  \mathbb{P}_{G \sim G_{n,p}}\left[E(G)=m\right]\leq O\left( \frac{1}{n^{\frac{C}{16}-2}}\right).
\end{equation}

or
\begin{equation}\label{finbound1} 
 \mathbb{P}_{G_1 \sim G(n,m)}\left[ \mathcal{H}^c_{\rho,C} \right] \leq  O\left(\frac{1}{n^{\frac{C}{16}-3}}\right).
\end{equation}
The proof of the Proposition is complete.
\end{proof}

\begin{lemma}\label{homog}
Let $C>1$. For any $G,G' \in \mathcal{H}_{\rho,C}$, $$\frac{1}{8C}\min\{ \frac{n^{\frac{3}{2}}}{\max\{\sqrt{\rho},\sqrt{\frac{\log n}{n}}\} \sqrt{\log n}}|e(G)-e(G')|,n\} \leq  d_V(G,G').$$ 
\end{lemma}

\begin{proof}
Let $G,G' \in \mathcal{H}_{\rho,C}=\bigcap_{\emptyset \not = S \subseteq [n]} A_{S,\rho,C}$. Since $\frac{1}{8C}<\frac{1}{4}$ we may assume without loss of generality that $d_V(G,G') < \frac{n}{4}$. It suffices to establish thatfor every pair $G,G' \in H_{\rho}$ with $\delta_V(G,G') < \frac{n}{4}$ it holds $$\frac{1}{4C}\frac{n^{\frac{3}{2}}}{\sqrt{\log n}}|e(G)-e(G')| \leq  \max\{\sqrt{\rho},\sqrt{\frac{\log n}{n}}\} d_V(G,G').$$

Let $k:=d_V(G,G')$ and assume without loss of generality $k \geq 1$. By definition of the node-distance there exist a non-empty subset of the vertices $S_0 \subseteq [n]$ with $|S_0|=k$ so that we can construct $G$ from $G'$ if we rewire only vertices belonging to the set $S_0$. From this property we conclude that the induced subgraphs defined only for the vertices of $S_0^c$  in $G$ and in $G'$ are isomorphic. In particular $$E_G(S_0^c,S_0^c)=E_{G'}(S_0^c,S_0^c)$$ and therefore
\begin{align*}|E(G)-E(G')|=|E_G(S_0,S_0)+E_G(S_0,S^c_0)-E_{G'}(S_0,S_0)-E_{G'}(S_0,S^c_0)|. \end{align*} Since $G,G' \in A_{\rho,S_0,C}$ we have \begin{align*} |E_G(S_0,S_0)+E_G(S_0,S^c_0)-e(G) \left( k(n-k)+\binom{k}{2} \right) |\leq C \max\{\sqrt{\rho},\sqrt{\frac{\log n}{n}}\} k  \sqrt{n \log n} \end{align*} and \begin{align*} |E_{G'}(S_0,S_0)+E_{G'}(S_0,S^c_0)-e(G') \left( k(n-k)+\binom{k}{2} \right) | \leq C \max\{\sqrt{\rho},\sqrt{\frac{\log n}{n}}\} k  \sqrt{n \log n}. \end{align*} Therefore by triangle inequality
\begin{align*}|E(G)-E(G')|\leq |e(G)-e(G')| \left( k(n-k)+\binom{k}{2} \right) +2C \max\{\sqrt{\rho},\sqrt{\frac{\log n}{n}}\} k  \sqrt{n \log n},\end{align*} which from the definition of edge density can be rewritten as
\begin{align*}\left( \binom{n}{2}-\left(k(n-k)+\binom{k}{2}\right) \right) |e(G)-e(G')|\leq 2C \max\{\sqrt{\rho},\sqrt{\frac{\log n}{n}}\} k  \sqrt{n \log n}.\end{align*}

Since  $ k(n-k)+\binom{k}{2} \leq kn$ we have
$$\left( \binom{n}{2}-kn \right)|e(G)-e(G')| \leq C \max\{\sqrt{\rho},\sqrt{\frac{\log n}{n}}\} k  \sqrt{n \log n}$$
But now as we have assumed $k \leq \frac{n}{4}-1$ we have $\binom{n}{2}-kn \geq \frac{n(n-1)}{2}-n(\frac{n}{4}-1) \geq \frac{n^2}{4}$ and therefore \begin{align*}\frac{n^2}{4}|e(G)-e(G')|\leq 2C \max\{\sqrt{\rho},\sqrt{\frac{\log n}{n}}\} k  \sqrt{n \log n}.\end{align*} or 
$$\frac{1}{8C}\frac{n^{\frac{3}{2}}}{\sqrt{\log n}}|e(G)-e(G')| \leq \max\{\sqrt{\rho},\sqrt{\frac{\log n}{n}}\} k    =\max\{\sqrt{\rho},\sqrt{\frac{\log n}{n}}\} \delta_V(G,G'),$$where in the last inequality we have used the definition of $k$.
The proof of the Lemma is complete.

\end{proof}

\paragraph{Construction of the Algorithm and First Results}

We now present the Algorithm that implies both the bounds of the Proposition \ref{prop:32-fixed-m} and the Theorem \ref{thm32}.

Let $C>48$. We first define the algorithm $\hat{\mathcal{A}}$ only for input graphs $G$ belonging in $\mathcal{H}_{\rho,C}$. For $G \in \mathcal{H}_{\rho,C}$, $\hat{\mathcal{A}}(G)$ samples from a continuous distribution on $[0,1]$ which adds \textit{truncated} Laplacian noise to the edge density. Specifically for $q \in [0,1]$ the density of the output distribution is given by $$f_{\hat{\mathcal{A}(G)}}(q)=\frac{1}{Z_{\hat{\mathcal{A}(G)}}}\exp \left(-\frac{\epsilon}{2} \frac{1}{8C} \min \{\frac{n^{\frac{3}{2}}}{\max\{\sqrt{\rho},\sqrt{\frac{\log n}{n}}\}\sqrt{\log n}} |e(G)-q|,n\}\right),$$where $$Z_{\hat{\mathcal{A}}(G)}=\int_{0}^1 \exp \left(-\frac{\epsilon}{2} \frac{1}{8C} \min \{\frac{n^{\frac{3}{2}}}{\max\{\sqrt{\rho},\sqrt{\frac{\log n}{n}}\}\sqrt{\log n}} |e(G)-q|,n\}\right)dq.$$

\begin{lemma}\label{restr} 
The algorithm $\hat{\mathcal{A}}$ defined for graphs from $\mathcal{H}_{\rho,C}$ is $\epsilon/2$-node-DP. Furthermore, for any $G \in \mathcal{H}_{\rho,C}$, 
\begin{equation} 
\bE\left[ \left(\hat{\mathcal{A}}(G)-e(G)\right)^2 \right]=O\left( \max\{\rho, \frac{\log n}{n} \} \frac{\log n}{n^3 \epsilon^2}\right).
\end{equation} 
\end{lemma}
\begin{proof}
Let $G,G' \in \mathcal{H}_{\rho,C}$.  Since both the distribution of $\hat{\mathcal{A}}(G)$ and $\hat{\mathcal{A}}(G')$ are continuous on $[0,1]$ and therefore absolutely continuous with respect to the Lebesqure measure on $[0,1]$ it suffices to check based on Lemma \ref{lem} that  \begin{equation}\label{check}  f_{\hat{\mathcal{A}}(G)}(q) \leq \exp \left(\frac{\epsilon}{2} d_V(G,G') \right) f_{\hat{\mathcal{A}}(G')}(q),  \end{equation}for all $q \in [0,1]$. Fix some $q \in [0,1]$. Then (\ref{check}) after taking logarithms and reaaranging is equivalent with \begin{align}\label{big}
\min \{\frac{n^{\frac{3}{2}}}{\max\{\sqrt{\rho},\sqrt{\frac{\log n}{n}}\}\sqrt{\log n}} |e(G')-q|,n\}   - \min \{\frac{n^{\frac{3}{2}}}{\max\{\sqrt{\rho},\sqrt{\frac{\log n}{n}}\}\sqrt{\log n}} |e(G)-q|,n\} 
\end{align} being upper bounded by  $8 C d_V(G,G')$. But Lemma \ref{triang} implies that (\ref{big}) is upper bounded by $$\min \{\frac{n^{\frac{3}{2}}}{\max\{\sqrt{\rho},\sqrt{\frac{\log n}{n}}\}\sqrt{\log n}} |e(G)-e(G')|,n\}$$Since $G,G' \in \mathcal{H}_{\rho,C}$ by Lemma \ref{homog} the last quantity is upper bounded by $8 C d_V(G,G')$, as we wanted.

For the second part, we first ease the notation by setting $r_n= \frac{n^{\frac{3}{2}}}{\max\{\sqrt{\rho},\sqrt{\frac{\log n}{n}}\}\sqrt{\log n}}$. Now we first bound the normalizing quantity $Z_{\hat{\mathcal{A}}(G)}$.

\begin{align*}
Z_{\hat{\mathcal{A}}(G)}&=\int_{0}^1 \exp \left(-\frac{\epsilon}{8C} \min \{r_n |e(G)-q|,n\}\right)dq \\
& \geq \int_{0}^1 \exp \left(-\frac{\epsilon}{16C} r_n |e(G)-q|\right)dq \\
& \geq  \frac{16C}{r_n \epsilon} \int_{e(G)-1}^{e(G)} \exp \left(- |u|\right)du \text{ (for }  u=\frac{\epsilon}{8C}r_n(e(G)-q) )\\
& \geq \frac{16C}{r_n \epsilon} \int_{-1}^1 \exp(-1)du\\
&=\Omega(\frac{1}{r_n \epsilon}).
\end{align*}
Therefore 

\begin{equation}\label{norm}
Z_{\hat{\mathcal{A}}(G)}=\Omega(\frac{1}{r_n \epsilon}). 
\end{equation}
Next we  bound the expected value as following:

\begin{align*}
\bE\left[ \left(\hat{\mathcal{A}}(G)-e(G)\right)^2 \right]&=\int_0^1 (q-e(G))^2 f_{\hat{\mathcal{A}}(G)}(q)dq\\
&= \frac{1}{Z_{\hat{\mathcal{A}}(G)}} \int_0^1 (q-e(G))^2\exp \left(-\frac{\epsilon}{8C}r_n |e(G)-q|,n\}\right)dq\\
& \leq \frac{1}{Z_{\hat{\mathcal{A}}(G)}} \int_{-1}^1 u^2\exp \left(-\frac{\epsilon}{8C}  \min \{r_n|u|,n\}\right)du, \text{ (for } u=q-e(G) )\\
& \leq \frac{1}{Z_{\hat{\mathcal{A}}(G)}} \left(  2\exp(-\frac{n\epsilon}{8C})+\int_{|u|<\min\{n/r_n
,1\}}u^2\exp \left(- \frac{\epsilon}{8C} r_n|u| \right)du \right)\\
& \leq \frac{1}{Z_{\hat{\mathcal{A}}(G)}} \left(  2\exp(-\frac{n\epsilon}{8C})+\int_{|u|<1}u^2\exp \left(- \frac{\epsilon}{8C} r_n|u| \right)du \right)\\
& =\frac{1}{Z_{\hat{\mathcal{A}}(G)}}\left(  2\exp(-\frac{n\epsilon}{8C})+O\left(\int_{|x|<r_n}\frac{x^2}{\epsilon^3 r_n^3}\exp \left(- |x| \right)du \right)  \right)  \text{ (for }x=\epsilon r_n u)\\
&=\frac{1}{Z_{\hat{\mathcal{A}}(G)}}O\left(  \exp(-\frac{n\epsilon}{8C})+\frac{1}{\epsilon^3 r_n^3} \right)\\
\end{align*}Using now the bound (\ref{norm}) ,

\begin{equation*}
\bE\left[ \left(\hat{\mathcal{A}}(G)-e(G)\right)^2 \right] =O\left(r_n \epsilon \exp(-\frac{n\epsilon}{8C})+\frac{1}{\epsilon^2 r_n^2} \right)
\end{equation*} Now we claim that because $\epsilon n \rightarrow + \infty$ we have $$ r_n \epsilon \exp(-\frac{n\epsilon}{8C}) =o(\frac{1}{\epsilon^2 r_n^2}).$$Indeed it suffices to show that $$r_n^3 \epsilon^3\exp(-\frac{n\epsilon}{8C})=o(1).$$Taking logarithms it suffices to show $$\log r_n+\log \epsilon \ll n\epsilon.$$ Since $r_n=O(n^{3})$, it suffices $$\log (n)\ll n \epsilon,$$ which is true as we assume $\epsilon n/\log n \rightarrow + \infty$.  Therefore
\begin{equation*}
\bE\left[ \left(\hat{\mathcal{A}}(G)-e(G)\right)^2 \right] =O\left(\frac{1}{\epsilon^2 r_n^2} \right).
\end{equation*}
Plugging in the value of $r_n$ we conclude
\begin{equation*}
\bE\left[ \left(\hat{\mathcal{A}}(G)-e(G)\right)^2 \right] =O\left( \max\{\rho, \frac{\log n}{n} \} \frac{\log n}{n^3 \epsilon^2}\right).
\end{equation*}
The proof of the Lemma is complete.
\end{proof}

Since by Lemma \ref{restr}, the algorithm $\hat{\mathcal{A}}$ is $\frac{\epsilon}{2}$-node-DP on graphs from $\mathcal{H}_{\rho,C}$, using Theorem \ref{extension} we can extend it to an algorithm $\mathcal{A}$ which is defined on every graph on $n$ vertices such that 
\begin{itemize}
\item  $\mathcal{A}$ is $\epsilon$-node-DP;
\item  For every $G \in \mathcal{H}_{\rho,C}$, $\mathcal{A}(G) \overset{d}{=}  \hat{\mathcal{A}}(G).$

\end{itemize}

This is the algorithm that we analyze to establish both Proposition \ref{prop:32-fixed-m} and Theorem \ref{thm32}.

\paragraph{Proof of Proposition \ref{prop:32-fixed-m} and Theorem \ref{thm32}}
\begin{proof}[Proof of Proposition \ref{prop:32-fixed-m}]
Fix $C>0$ large enough constant ( $C>48$ suffices for our initial choice).
We claim that $\mathcal{A}$ satisfies the necessary property. Fix $m<\rho \binom{n}{2}$.  We first split the expected squared error depending on whether the samples graph depends on $\mathcal{H}_{\rho,C}$ or not.
\begin{align*} 
&\bE_{G \sim G(n,m)}\left[ \left(\hat{\mathcal{A}}(G)-e(G)\right)^2 \right]\\
&=\bE_{G \sim G(n,m)}\left[ \left(\hat{\mathcal{A}}(G)-e(G)\right)^2 1(G  \not \in H_{\rho,C}) \right]+\bE_{G \sim G(n,m)}\left[ \left(\hat{\mathcal{A}}(G)-e(G)\right)^2 1(G \in H_{\rho,C}) \right]
\end{align*}For the first term we have from Lemma \ref{union} that
\begin{equation}
\bE_{G \sim G(n,m)}\left[ \left(\hat{\mathcal{A}}(G)-e(G)\right)^2 1(G  \not \in \mathcal{H}_{\rho,C}) \right] \leq \mathbb{P}_{G \sim G(n,m)}\left[ G \not \in \mathcal{H}_{\rho,C}\right] =O\left( n^{-(\frac{C}{16}-3)}\right).
\end{equation}
For the second term we have from Lemma \ref{homog} that
\begin{align}
\bE_{G \sim G(n,m)}\left[ \left(\hat{\mathcal{A}}(G)-e(G)\right)^2 1(G  \in \mathcal{H}_{\rho,C}) \right] &\leq \max_{G \in H}\bE_{G \sim G(n,m)}\left[ \left(\hat{\mathcal{A}}(G)-e(G)\right)^2  \right] \\
&\leq O\left( \max\{\rho, \frac{\log n}{n} \} \frac{\log n}{n^3 \epsilon^2}\right).
\end{align}Combining the above we conclude that for any $C>48$,
\begin{equation*}
\bE_{G \sim G(n,m)}\left[ \left(\hat{\mathcal{A}}(G)-e(G)\right)^2 \right] =O\left(n^{-(\frac{C}{16}-3)}+ \max\{\rho, \frac{\log n}{n} \} \frac{\log n}{n^3 \epsilon^2}\right).
\end{equation*}Since $\epsilon<1$, by choosing $C$ sufficiently large but constant we conclude \begin{equation*}
\bE_{G \sim G(n,m)}\left[ \left(\hat{\mathcal{A}}(G)-e(G)\right)^2 \right] =O\left( \max\{\rho, \frac{\log n}{n} \} \frac{\log n}{n^3 \epsilon^2}\right).
\end{equation*}This completes the proof.
\end{proof}

\begin{proof}[Proof of Theorem \ref{thm32}]
Fix a $C>0$ large enough constant ($C>48$ suffices as an initial choice).
To prove the upper bound on the rate we discuss the performance of the algorithm $\mathcal{A}$ defined above. To bound its mean squared error consider a $p \in [0,\rho].$ We first use the bias-variance decomposition to get,
\begin{align*}
\bE_{G \sim G_{n,p}}\left[ \left(\hat{\mathcal{A}}(G)-p\right)^2 \right]& = \bE_{G \sim G_{n,p}}\left[ \left(\hat{\mathcal{A}}(G)-e(G)\right)^2 \right]+\bE_{G \sim G_{n,p}}\left[ \left(e(G)-p\right)^2 \right] .
\end{align*}The second term is the variance of the edge density and therefore \begin{equation}\bE_{G \sim G_{n,p}}\left[ \left(e(G)-p\right)^2 \right] =O\left( \frac{p}{n^2}\right)=O\left( \frac{\rho}{n^2}\right) .\end{equation}For the first term, we recall that a sample from $G_{n,p}$ conditional on having a fixed number of edges $m$, is distributed according to $G(n,m)$. Therefore 
\begin{equation}
\bE_{G \sim G_{n,p}}\left[ \left(\mathcal{A}(G)-e(G)\right)^2 \right] =\bE_{G \sim G_{n,p}}\left[ \bE_{G' \sim G(n,E(G))}\left[\left(\mathcal{A}(G')-e(G')\right)^2\right] \right].
\end{equation}Using Proposition \ref{prop:32-fixed-m} we conclude \begin{equation}
\bE_{G \sim G_{n,p}}\left[ \left(\mathcal{A}(G)-e(G)\right)^2 \right] =O\left(n^{-(C/16-3)}+ \max\{\rho, \frac{\log n}{n} \} \frac{\log n}{n^3 \epsilon^2}\right).
\end{equation}Combining the above we conclude that for any $C>48$, \begin{align*}
\bE_{G \sim G_{n,p}}\left[ \left(\hat{\mathcal{A}}(G)-p\right)^2 \right]=O\left(\frac{\rho}{n^2}+ \max\{\rho, \frac{\log n}{n} \} \frac{\log n}{n^3 \epsilon^2}+n^{-(C/16-3)}\right). \end{align*}Since $\epsilon<1$ by choosing $C>0$ sufficiently large but constant we conclude \begin{align*}
\bE_{G \sim G_{n,p}}\left[ \left(\hat{\mathcal{A}}(G)-p\right)^2 \right]=O\left(\frac{\rho}{n^2}+ \max\{\rho, \frac{\log n}{n} \} \frac{\log n}{n^3 \epsilon^2}\right). \end{align*} The proof of Theorem \ref{thm32} is complete.
\end{proof}

\section{Proof of Lemma \ref{Lap}}\label{Lap}
Using Lemma 10 from \cite{BorgsCS15} we have that the estimator $f(G)=e(G)+Z$, for $Z$ following $\mathrm{Lap}(\frac{4}{n \epsilon})$, is $\epsilon$-node-DP. Therefore $$R_1(\rho,\epsilon,n) \leq \mathbb{E}_{G \sim G_{n,p}}[  (f(G)-p)^2] =O(\mathbb{E}_{G \sim G_{n,p}}[(e(G)-p)^2]+\mathbb{E}_{G \sim G_{n,p}}[Z^2].$$ The first term is the variance of the edge density and therefore $$\mathbb{E}_{G \sim G_{n,p}}[(e(G)-p)^2]=O(\frac{\rho}{n^2}).$$For the second term, we clearly have that it is of the order $O(\frac{1}{n^2 \epsilon^2})$. Therefore we conclude 
$$\mathbb{E}_{G \sim G_{n,p}}[  (f(G)-p)^2] =O(\frac{\rho}{n^2}+\frac{1}{n^2 \epsilon^2}.)$$The proof is complete.

\section{Proofs for Proposition \ref{coupl} and Theorem \ref{thmcoupl}}\label{LB}
\paragraph{Auxilary Lemmata}

\begin{lemma}\label{binom2}
For $a,b,k \in \mathbb{N}$ with $a,b \rightarrow +\infty$ and $k=o(\min\{a,b\})$, $$ \binom{a}{k}/\binom{b}{k}=\Theta\left(\left(\frac{a}{b}\right)^k \right).$$
\end{lemma}

\begin{proof}
By Stirling approximation we have $n!=\Theta((\frac{n}{e})^n\sqrt{n})$. Therefore as
 \begin{align*} 
\binom{a}{k}/\binom{b}{k}&=\left(a! (b-k)!\right)/\left(b! (a-k)! \right) \\
&=\Theta( \frac{ a^a (b-k)^{b-k}}{b^b (a-k)^{a-k}}\sqrt{ \frac{(a-k)b}{a(b-k)}})\\
&= \Theta \left( (\frac{a}{b})^k (1+\frac{k}{a-k})^{a-k-\frac{1}{2}}  (1+\frac{k}{b-k})^{-(b-k)-\frac{1}{2}} \right).
\end{align*}Since $k=o(\min\{a,b\})$, $(1+\frac{k}{a-k})^{a-k-\frac{1}{2}} =e+o(1)$ and $(1+\frac{k}{b-k})^{-(b-k)-\frac{1}{2}}=e^{-1}+o(1)$. Combining them we conclude $$ (1+\frac{k}{a-k})^{a-k-\frac{1}{2}}  (1+\frac{k}{b-k})^{-(b-k)-\frac{1}{2}} =1+o(1)$$ or $$\Theta \left( (\frac{a}{b})^k (1+\frac{k}{a-k})^{a-k-\frac{1}{2}}  (1+\frac{k}{b-k})^{-(b-k)-\frac{1}{2}}\right) =\Theta(\frac{a}{b})^k).$$  The proof is complete.
\end{proof}

\begin{lemma}\label{dmin}
Let $N=\binom{n}{2}$ and $N \geq m \geq N/3$.
Let $G_0$ sampled from $P=G(n,m).$  $$\mathbb{P}_{G_0 \sim P}\left(d^{G_0}_{\min} \geq n/5 \right) =1-2^{-\Omega(n)}.$$
\end{lemma}

\begin{proof}
It suffices to establish that for $v \in [n]$ arbitrary vertex in the graph, $$\mathbb{P}_{G_0 \sim P}\left(d^{G_0}(v) \leq n/5 \right) =2^{-\Omega(n)}.$$Then the proof follows by a union bound over the set of vertices. By Pittel's inequality (see Section 1.4 in \cite{Jan2011}) we can upper bound the probability by the corresponding probability for the Erdos-Renyi model multiplied by $O(\sqrt{m})$,  $$\mathbb{P}_{G_0 \sim P}\left(d^{G_0}(v) \leq n/5 \right) \leq 3 \sqrt{m} \mathbb{P}_{G_0 \sim G_{n,m/N}}\left(d^{G_0}(v) \leq n/5 \right).$$ Since $m<n^2$, it suffices to establish$$\mathbb{P}_{G_0 \sim G_{n,m/N}}\left(d^{G_0}(v) \leq n/5 \right) =2^{-\Omega(n)}.$$The distribution, though, of $d^{G_0}(v)$ in the Erdos Renyi model is a binomial with parameters $n-1,p=m/N>1/3$. In particular, Hoeffding's inequality implies $$\mathbb{P}_{G_0 \sim G_{n,m/N}}\left(d^{G_0}(v) \leq n/5 \right)  \leq \mathbb{P}_{G_0 \sim G_{n,m/N}}\left(|d^{G_0}(v)-\frac{m}{N}(n-1)| \geq \frac{2}{15}n \right) =2^{-\Omega(n)}.$$The proof is complete.
\end{proof}
\begin{proof}[Proof of Proposition \ref{coupl}] We consider two models. The first is $\mathbb{P}_1=P=G(n,m)$, that is a sample of a uniform graph on $n$ vertices and $m$ edges. The second is $\mathbb{P}_2=G(n,m,k)$: sample first uniformly a graph according to $Q$, that is a uniform graph on $n$ vertices and $m+k$ edges. Then choose a uniformly chosen vertex $v$ of the graph and delete $\min\{ d^{G_0}_v,k\}$ edges which are adjacent to the vertex, uniformly at random. Note that by $d^{G_0}_v$ we refer to the degree of the vertex in $G_0$. We claim that under the assumptions of our Proposition, \begin{equation}\label{targetfinnn}
\lim_{n \rightarrow +\infty} \mathrm{TV}(\mathbb{P}_1,\mathbb{P}_2)=0.
\end{equation} 
Note that after proving this we are done for the following reason. First, it implies that with probability tending to one there is a coupling between $\mathbb{P}_1$ and $\mathbb{P}_2$ such that they output the same graph with probability tending to one. Since $\mathbb{P}_1=P$ and $\mathbb{P}_2$ samples a graph from $Q$ and rewires a single vertex of the output graph, we conclude that there exists a coupling $(G,H)$ coming from $P$ and $Q$ respectively such that, with probability tending to one, one can obtain $G$ from $H$ by rewiring one vertex.

For the proof of (\ref{targetfinnn}) by Pinsker's inequality we have that $$\mathrm{TV}(\mathbb{P}_1,\mathbb{P}_2) \leq \sqrt{\mathrm{KL}(\mathbb{P}_1,\mathbb{P}_2)}.$$Therefore, it suffices to prove that the KL divergence converges to zero or, $$ \limsup_{n \rightarrow + \infty} \mathbb{E}_{G_0 \sim \mathbb{P}_1} \left[ \log \frac{\mathbb{P}_1[G=G_0]}{\mathbb{P}_2[G=G_0]}  \right]=0.$$
For convenience, we focus on the equivalent $$ \liminf_{n \rightarrow + \infty} \mathbb{E}_{G_0 \sim \mathbb{P}_1} \left[ \log \frac{\mathbb{P}_2[G=G_0]}{\mathbb{P}_1[G=G_0]}  \right]=0.$$By Jensen's inequality since $\log $ is concave, we have for all $n$, $$\mathbb{E}_{G_0 \sim \mathbb{P}_1} \left[ \log \frac{\mathbb{P}_2[G=G_0]}{\mathbb{P}_1[G=G_0]}  \right] \leq \log \mathbb{E}_{G_0 \sim \mathbb{P}_1} \left[  \frac{\mathbb{P}_2[G=G_0]}{\mathbb{P}_1[G=G_0]}  \right]=\log 1=0.$$Therefore it suffices to show  $$ \liminf_{n \rightarrow + \infty} \mathbb{E}_{G_0 \sim \mathbb{P}_1} \left[ \log \frac{\mathbb{P}_2[G=G_0]}{\mathbb{P}_1[G=G_0]}  \right] \geq 0.$$

Now for any $G_0$ on $n$ vertices with $m$ edges we lower bound $\mathbb{P}_2[G=G_0]$ as follows,
 
\begin{align*}
\mathbb{P}_2[G=G_0]&=\sum_{G' \text{ with m+k edges}} \mathbb{P}(G' \text{ is chosen in the first step}) \mathbb{P}(G_0|G')\\
&=\sum_{G' \text{ with m+k edges and }\mathbb{P}(G_0|G')>0 } \frac{1}{\binom{N}{M+k}} \mathbb{P}(G_0|G'), \text{ ($G'$ is chosen according to $Q$	.)}\\
&= \frac{1}{\binom{N}{m+k}} \sum_{v \in V(G_0) } \sum_{G' \text{ is plausible by } G_0 \text{ via rewiring} v}\mathbb{P}(G_0|G'),  \text{  (plausible refers to  } \mathbb{P}(G_0|G')>0 )  \\
&=  \frac{1}{\binom{N}{m+k}} \sum_{v \in V(G_0) } \sum_{G' \text{ is plausible by } G_0 \text{ via rewiring } v}\frac{1}{n\binom{d^{G_0}(v)+k}{k}}\\
&=  \frac{1}{n\binom{N}{m+k}} \sum_{v \in V(G_0)} \frac{\binom{n-d^{G_0}(v)-1}{k}}{\binom{d^{G_0}(v)+k}{k}}
\end{align*}
Since $\mathbb{P}_1[G=G_0]=\frac{1}{\binom{N}{m}}$, the ratio is equal to 
\begin{equation}\label{ratio} \frac{\mathbb{P}_2[G=G_0]}{\mathbb{P}_1[G=G_0]} =\frac{1}{n} \sum_{v \in V(G_0)} \frac{\binom{n-d^{G_0}(v)-1}{k}}{\binom{d^{G_0}(v)+k}{k}}.\end{equation}   

 Set $\mathcal{E}=\{d^{G_0}_{\min} \geq n/5\}$.
 We claim that it suffices to show \begin{equation}\label{suff} \liminf_{n \rightarrow + \infty} \mathbb{E}_{G_0 \sim \mathbb{P}_1} \left[ \log \frac{\mathbb{P}_2[G=G_0]}{\mathbb{P}_1[G=G_0]}  | \mathcal{E}\right]\geq 0.\end{equation} Indeed since for any $G_0$ on $m$ edges (\ref{ratio}) holds, $ \log \frac{\mathbb{P}_2[G=G_0]}{\mathbb{P}_1[G=G_0]}$ is at most a quantity which is polynomial in $n$. Therefore  $$ \mathbb{E}_{G_0 \sim \mathbb{P}_1} \left[ \log \frac{\mathbb{P}_2[G=G_0]}{\mathbb{P}_1[G=G_0]}  \right] \geq \mathbb{P}\left( \mathcal{E}^c \right) \mathrm{poly}(n)+\mathbb{P}\left( \mathcal{E} \right)\mathbb{E}_{G_0 \sim \mathbb{P}_1} \left[ \log \frac{\mathbb{P}_2[G=G_0]}{\mathbb{P}_1[G=G_0]} | \mathcal{E} \right] .$$By Lemma \ref{dmin} we have $\mathbb{P}\left( \mathcal{E}^c \right)=2^{-\Omega(n)}$, and therefore indeed the condition (\ref{suff}) suffices for our result.

Now conditioning on $G_0 \in \mathcal{E}$ we compute,
\begin{align*}
&\frac{\mathbb{P}_2[G=G_0]}{\mathbb{P}_1[G=G_0]}=   \frac{1}{n} \sum_{v \in V(G_0)} \left(\frac{n-d^{G_0}(v)-1}{d^{G_0}(v)+k}\right)^{k}, \text{ (by Lemma \ref{binom2} )}\\
&\geq   \frac{1}{n} \sum_{v \in V(G_0)} \left(1+k\frac{n-2d^{G_0}(v)-k-1}{d^{G_0}(v)+k}\right), \text{ (using } (1+x)^k \geq 1+xk, x>-1)
\end{align*}
Therefore taking logarithms and conditional expectation it suffices to show  \begin{equation}\label{secondh} \liminf_{n} \mathbb{E}_{G_0 \sim \mathbb{P}_1}[\log \left(\frac{1}{n} \sum_{v \in V(G_0)} \left(1+k\frac{n-2d^{G_0}(v)-k-1}{d^{G_0}(v)+k}\right) \right) \bigg{|} \mathcal{E}] \leq 0\end{equation}We use Jensen's inequality and that $\log$ is concave to conclude 
\begin{align*} 
 \mathbb{E}_{G_0 \sim \mathbb{P}_1}[\log \left(\frac{1}{n} \sum_{v \in V(G_0)} \left(1+k\frac{n-2d^{G_0}(v)-k-1}{d^{G_0}(v)+k}\right) \right)\bigg{|} \mathcal{E}]\end{align*} is at least \begin{align*} \frac{1}{n} \sum_{v \in V(G_0)} \mathbb{E}_{G_0 \sim \mathbb{P}_1}[\log  \left(1+k\frac{n-2d^{G_0}(v)-k-1}{d^{G_0}(v)+k}\right)\bigg{|} \mathcal{E} ]
\end{align*}Since the process is node-symmetric we conclude for an arbitrary fixed vertex $v$,
\begin{equation*} 
\mathbb{E}_{G_0 \sim \mathbb{P}_1}[\log \left(\frac{1}{n} \sum_{v \in V(G_0)} \left(1+k\frac{n-2d^{G_0}(v)-k-1}{d^{G_0}(v)+k}\right) \right)\bigg{|}\mathcal{E}] \geq \mathbb{E}_{G_0 \sim \mathbb{P}_1}[\log  \left(1+k\frac{n-2d^{G_0}(v)-k-1}{d^{G_0}(v)+k}\right)\bigg{|} \mathcal{E}].
\end{equation*}Therefore it suffices to show
 \begin{equation*}
\lim_{n \rightarrow + \infty}\mathbb{E}_{G_0 \sim \mathbb{P}_1}[|\log  \left(1+k\frac{n-2d^{G_0}(v)-k-1}{d^{G_0}(v)+k}\right) | \bigg{|} \mathcal{E}] =0.
 \end{equation*}Using $\log (1+x) \leq x$ for $x>-1$, it suffices to show
\begin{equation*}
\lim_{n \rightarrow + \infty}\mathbb{E}_{G_0 \sim \mathbb{P}_1}[|  \left(k\frac{n-2d^{G_0}(v)-k-1}{d^{G_0}(v)+k}\right) | \bigg{|} \mathcal{E}] =0.
 \end{equation*}Since we condition on the minimum degree being at least of order $n$, we have\begin{align*}\mathbb{E}_{G_0 \sim \mathbb{P}_1}[|  \left(k\frac{n-2d^{G_0}(v)-k-1}{d^{G_0}(v)+k}\right) |\bigg{|} \mathcal{E}]& \leq O \left( \mathbb{E}_{G_0 \sim \mathbb{P}_1}[|  \left(k\frac{n-2d^{G_0}(v)-k-1}{n}\right) |\bigg{|} \mathcal{E}]\right) \\
& =O \left( \frac{k}{n}\mathbb{E}_{G_0 \sim \mathbb{P}_1}[|n-2d^{G_0}(v)-k-1| |\bigg{|} \mathcal{E}]\right) \\
& \leq O \left(\frac{k^2}{n}+ \frac{k}{n}\mathbb{E}_{G_0 \sim \mathbb{P}_1}[|n-2d^{G_0}(v) |\bigg{|} \mathcal{E}]\right)
\end{align*}Since $k=o(\sqrt{n})$ it suffices to prove \begin{equation}\label{finn} \lim_{n} \frac{k}{n}\mathbb{E}_{G_0 \sim \mathbb{P}_1}[|n-2d^{G_0}(v) |\bigg{|} \mathcal{E}] =0. \end{equation}Now \begin{equation*} \frac{k}{n}\mathbb{E}_{G_0 \sim \mathbb{P}_1}[|n-2d^{G_0}(v) |] =\mathbb{P}(\mathcal{E})\frac{k}{n}\mathbb{E}_{G_0 \sim \mathbb{P}_1}[|n-2d^{G_0}(v) |\bigg{|} \mathcal{E}]+\mathbb{P}(\mathcal{E}^c)\frac{k}{n}\mathbb{E}_{G_0 \sim \mathbb{P}_1}[|n-2d^{G_0}(v) |\bigg{|} \mathcal{E}^c]. \end{equation*}Since by Lemma \ref{dmin}, $\mathbb{P}(\mathcal{E}^c)=2^{-\Omega(n)}$ and $|n-2d^{G_0}(v) |=O(n),$almost surely, we conclude that to prove (\ref{finn}) it suffices to prove  \begin{equation}\label{finnn} \lim_{n} \frac{k}{n}\mathbb{E}_{G_0 \sim \mathbb{P}_1}[|n-2d^{G_0}(v) ] =0. \end{equation}But the degree of a vertex of a uniform random graph with $m$ edges follows an hypergeometric distribution with population size $N$, $n-1$ success states and $m$ number of draws. In particular, it's mean is $\frac{m(n-1)}{N}=\frac{(n-1)(N/2-k)}{N}=\frac{n-1}{2}(1+O(\frac{k}{N}))$ and variance $O \left( \frac{mnN^2}{N^3} \right)=O \left( \frac{mn}{N} \right)=O(n).$ Therefore by triangle inequality,
\begin{align*} 
\frac{k}{n}\mathbb{E}_{G_0 \sim \mathbb{P}_1}[|n-2d^{G_0}(v) |] & \leq \frac{k}{n} O\left(|n/2- \mathbb{E}_{G_0 \sim \mathbb{P}_1}(d^{G_0}(v))|+\mathbb{E}_{G_0 \sim \mathbb{P}_1}[|\mathbb{E}_{G_0 \sim \mathbb{P}_1}(d^{G_0}(v))-(d^{G_0}(v)) |]  \right)\\
& \leq \frac{k}{n}  O\left(|n/2- \frac{n-1}{2}(1+O(\frac{k}{N})|+\sqrt{\mathrm{Var}(d^{G_0}(v))} \right), \text{ (by Cauchy-Scharwz)}\\
&= \frac{k}{n} O \left(\frac{nk}{N}+\sqrt{n}  \right)\\
&= O \left( \frac{k^2}{N}+\frac{k}{\sqrt{n}} \right)\\
&= o(1) \text{ (since  } N=\Theta(n^2), k^2=o(\sqrt{n}))
\end{align*} The proof is complete.
\end{proof}

\begin{proof}[Proof of Theorem \ref{thmcoupl}]Set $\beta(\epsilon)=\exp(\epsilon)-1$.

Assume that there exists an $\epsilon$-node-DP algorithm, $\mathcal{A}$, which can distinguish $P=G(n,m)$ and $Q=G(n,m+k)$ with probability bigger than $\exp(\epsilon)-1>0$.  Since $\mathcal{A}$ can distinguish between the models $P,Q$ there exists a query set $S$ such that for some $\delta>\exp(\epsilon)-1$, \begin{equation}\label{test}\liminf_n|\mathbb{P}_{G \sim P}\left(\mathcal{A}(G) \in S\right)-\mathbb{P}_{H \sim Q}\left(\mathcal{A}(H) \in S\right) |\geq \delta>0.\end{equation} Now in Proposition \ref{coupl} we defined a disribution $R$ on unidrected graphs on $n$ vertices such that if $T$ is sampled from $R$ and $H$ from $Q$, $d_v(T,H)=1$, and furthermore \begin{equation}\label{tv1} \mathrm{TV}(R,P) =o(1).\end{equation}Since $\mathcal{A}$ is $\epsilon$-node-DP,  for any query $S_0$, $$\exp(-\epsilon)\leq \frac{\mathbb{P}_{T \sim R}\left(\mathcal{A}(T) \in S_0\right)}{\mathbb{P}_{H \sim Q}\left(\mathcal{A}(H) \in S_0\right)} \leq \exp(\epsilon).$$In particular, that implies \begin{equation}\label{tv2}\mathrm{TV}\left(R,Q\right) \leq \exp(\epsilon)-1.\end{equation}

Using (\ref{tv1}), (\ref{tv2}) and triangle inequality we obtain $$\mathrm{TV}(Q,P) \leq o(1)+\exp(\epsilon)-1.$$

Combined with (\ref{test}) we conclude $\delta \leq \exp(\epsilon)-1,$ a contradiction.

\end{proof}
\fi

\section*{Acknowledgments}

A.S. was supported by NSF awards IIS-1447700 and AF-1763665, and a
Sloan Foundation Research Award. I.Z. would like to thank Microsoft Research New England for providing exciting and hospitable enviroment during his summer internship in 2017 where part of this work was conducted.

\bibliographystyle{abbrvnat}
\bibliography{Lipschitz-extensions}

\end{document}